\documentclass[11pt]{article}
\usepackage{amsfonts}
\usepackage{mathrsfs}
\usepackage{amssymb,amsmath,mathrsfs, amsthm,amsfonts}
\usepackage{palatino,cite}
\usepackage{graphicx}
\usepackage{mathtools}
\usepackage{hyperref,cases,anysize,enumerate}

\numberwithin{equation}{section}

\newtheorem{theorem}{Theorem}[section]
\newtheorem{lemma}{Lemma}[section]
\newtheorem{corollary}{Corollary}[section]
\newtheorem{proposition}{Proposition}[section]
\theoremstyle{definition}
\newtheorem{definition}{Definition}[section]
\newtheorem{remark}{Remark}[section]
\newtheorem{example}{Example}[section]

\theoremstyle{remark}

\date{}
\begin{document}

\title{A Schwarz lemma for weakly K\"ahler-Finsler manifolds}
\author{Jun Nie (jniemath@126.com)\\
School of Mathematical Sciences, Xiamen
University\\ Xiamen 361005, China\\
Chunping Zhong (zcp@xmu.edu.cn)\\
School of Mathematical Sciences, Xiamen
University\\ Xiamen 361005, China
}

\date{}
\maketitle

\begin{abstract}
In this paper, we first establish  several theorems about the estimation of distance function on real and strongly convex complex Finsler manifolds and then obtain a Schwarz lemma from a strongly convex weakly K\"ahler-Finsler  manifold into a strongly pseudoconvex complex Finsler manifold. As applications, we prove that a holomorphic mapping from a strongly convex weakly K\"ahler-Finsler manifold into a strongly pseudoconvex complex Finsler manifold is necessary constant under an extra condition. In particular, we prove that a holomorphic mapping from a complex Minkowski space into a strongly  pseudoconvex complex Finsler manifold such that its holomorphic sectional curvature is bounded from above by a negative constant is necessary constant.
\end{abstract}



\section{Introduction and main results}
In complex analysis, the classical Schwarz-Pick lemma \cite{pick} states that any holomorphic mapping from the unit disk into itself decreases the Poincar\'e metric. There are various kinds of generalizations of the Schwarz-Pick lemma  among which the most influential one was given by Ahlfors \cite{ahlfors}, where he generalized the Schwarz lemma to holomorphic mappings from the unit disk $D$ into a Riemann surface $S$ endowed with a Riemannian metric $ds^2$ with Gauss curvature $K\leq -1$, and proved that the hyperbolic length of any curve in $D$ is at least equal to the length of its image.
Indeed, Ahlfors revealed a fundamental fact that the Schwarz lemma is closely related to the metric geometry of the domain and target manifolds, thus opened the door of generalizing Schwarz lemma from the view point of differential geometry.
The key idea used by Ahlfors in his proof is to compare the pull-back metric under the conformal map with the original hyperbolic metric on the unit disk, and then use the fact that the Laplacian of a real function must be nonnegative at the point where it attains a local minimum.

 Ahlfors' Schwarz lemma was later generalized by many mathematicians. In 1957-1958, Look \cite{La,Lb} gave a systematic study of Schwarz lemma and analytic invariants on the classic domains, from the viewpoints of both function theoretic and differential geometric.
 The Schwarz lemma has become a powerful tool in geometry and analysis ever since Yau's seminal paper \cite{yau2} which pushed this classic result in complex analysis to manifolds.
  The general theme of the lemma goes something like this: given a holomorphic map $f$ from a complete complex manifold $M$ into a target complex manifold $N$, assume that $M$ has lower curvature bound by a constant $K_1$ and $N$ has upper curvature bound by a negative constant $K_2<0$. Then the pull-back via $f$ of the metric of $N$  is dominated by a multiple (which is typically in the form $\frac{K_1}{K_2}$) of the metric of $M$, with the multiple given by the curvature bounds. This type of results immediately imply Liouville type rigidity results when the multiple becomes zero.

  In Yau's original result \cite{yau2},$M$ is a complete K\"ahler manifold and $N$ is another Hermitian manifold, where $M$ has Ricci curvature bounded from below by a constant $K_1$ and $N$ has holomorphic bisectional curvature bounded from above by a negative constant $K_2 <0$. Shortly after, Rodyden \cite{royden} realized that the holomorphic bisectional curvature assumption on the domain could be reduced to holomorphic sectional utilizing the symmetry of curvature tensor of K\"ahler metrics. Since then, various generalizations were made to Hermitian and almost Hermitian cases, we refer to Chen-Cheng-Lu \cite{chen}, Greene-Wu \cite{greene}, Lu-Sun \cite{LS}, Liu \cite{Liu}, Zuo \cite{zuo}, Tosatti \cite{tosatti}, Wu-Yau \cite{wy}, Yang-Zheng \cite{yang}, Ni \cite{Nia,Nib} and many others. We also refer to Kim and Lee \cite{KL} and references therein. Note that recently the study of the Schwarz lemma at the boundary of various type of domains in $\mathbb{C}^n$ also attracts lots of interests, we refer to Burns-Krantz \cite{BK}, Liu-Tang \cite{Liua,Liub,Liuc, Liud}, Tang-Liu-Lu \cite{Tang}, Tang-Liu-Zhang\cite{Tangb},  Wang-Liu-Tang \cite{Wang}, etc.

In \cite{chen}, Chen, Cheng and Lu established the following  Schwarz lemma.
\begin{theorem}(\cite{chen})\label{chen}
Suppose that $M$ is a complete K\"ahler manifold such that its holomorphic sectional curvature is bounded from below by $K_1$ and that its sectional curvature is also bounded from below. Suppose that $N$ is a Hermitian manifold whose holomorphic sectional curvature is bounded from above by a negative constant $K_2$. Let $f:M\rightarrow N$ be any holomorphic mapping. Then
\begin{equation}
f^\ast ds_N^2\leq \frac{K_1}{K_2}ds_M^2.
\end{equation}
\end{theorem}

Note that both the metrics $ds_M^2$ and $ds_N^2$ in Theorem \ref{chen} are Hermitian quadratic metrics.
According to S. S. Chern \cite{Ch}, Finsler geometry is just Riemannian geometry without quadratic restrictions \cite{Ch}. Thus complex Finsler geometry is just Hermitian geometry without Hermitian quadratic restrictions which contains Hermitian geometry as special case.
It is known that on any complex manifold, there are natural intrinsic pseudo-metric, i.e., the Kobayashi pseudo-metric \cite{Kobayashi3} and the Carath$\acute{\mbox{e}}$odory pseudo-metric \cite{C}.  In general, however, they are only complex Finsler metrics in nature and in some special cases these metrics are even strongly pseudoconvex complex Finsler metrics in the strict sense of Abate and Patrizio \cite{abate}.

There are three notions of K\"ahlerian in complex Finsler setting \cite{abate}, that is, strongly K\"ahler-Finsler metric, K\"ahler-Finsler metric and weakly K\"ahler-Finsler metric. It is proved by Chen and Shen \cite{chenb} that a K\"ahler-Finsler metric is actually a strongly K\"ahler-Finsler metric, thus leaving two notions of K\"aherian in complex Finsler setting. There are, however, lots of nontrivial K\"ahler-Finsler metrics (that is, they are neither complex Minkowski metrics nor K\"ahler metrics). Indeed,
 let $\pmb{\alpha}^2(\xi)=a_{i\overline{j}}(z)\xi^i\overline{\xi^j}$ and $\pmb{\beta}^2(\eta)=b_{i\overline{j}}(w)\eta^i\overline{\eta^j}$ be two Hermitian metrics on complex manifolds $M_1$ and $M_2$, respectively. In \cite{XZ}, Xia and Zhong proved that the following Szab$\acute{\mbox{o}}$ metric
\begin{equation}
F_\varepsilon=\sqrt{\pmb{\alpha}^2(\xi)+\pmb{\beta}^2(\eta)+\varepsilon(\pmb{\alpha}^{2k}(\xi)+\pmb{\beta}^{2k}(\eta))^{\frac{1}{k}}},\quad \varepsilon\in(0,+\infty)\label{kfm}
\end{equation}
is actually a strongly convex complex Berwald metric on the product manifold $M=M_1\times M_2$ for any $k\geq 2$. Moreover, they proved that $F_\varepsilon$ is a strongly convex K\"ahler-Finsler metric on  $M=M_1\times M_2$ if and only if both $(M_1,\pmb{\alpha})$ and $(M_2,\pmb{\beta})$ are K\"ahler manifolds. Note that the metrics $F_\varepsilon$ defined by \eqref{kfm} are non-Hermitian quadratic for any $\varepsilon\in(0,+\infty)$ and integer $k\geq 2$.

For weakly K\"ahler-Finsler metrics, we need to mention the fundamental theorem of Lempert \cite{lempert}, which  states that on any bounded strongly convex domain $D\subset\subset\mathbb{C}^n$ with smooth boundary, the Kobayashi pseudo-metric and the  Carath$\acute{\mbox{e}}$odory pseudo-metric coincide, and they are strongly pseudoconvex complex Finsler metrics in the sense of Abate and Patrizio \cite{abate}, namely, they are smooth outside of the zero section of the holomorphic tangent bundle $T^{1,0}D$. Moreover, they are weakly K\"ahler-Finsler metric with constant holomorphic sectional curvature $-4$. In general, however, these metrics do not have explicit formulas on strongly convex domains with smooth boundaries in $\mathbb{C}^n$. It is still open whether the Kobayashi metrics on such domains are K\"ahler-Finsler metrics, or more specifically, K\"ahler-Berwald metrics?

To construct strongly pseudoconvex complex Finsler metrics with specific properties, Zhong \cite{Zh1} initiated the study of $U(n)$-invariant complex Finsler metrics $F(z,v)=\sqrt{r\phi(t,s)}$ on $U(n)$-invariant domains $D\subseteq\mathbb{C}^n$, here
$$r=\|v\|^2,\quad t=\|z\|^2,\quad s=\frac{|\langle z,v\rangle|^2}{r},\quad \forall (z,v)\in T^{1,0}D,$$
and proved that a $U(n)$-invariant complex Finsler metric $F=\sqrt{r\phi(t,s)}$ is a K\"ahler-Finsler metric if and only if $\phi(t,s)=f(t)+f'(t)s$ for some positive function $f(t)$ satisfying $f(t)+tf'(t)>0$. It was also proved in \cite{Zh1} that a strongly pseudoconvex $U(n)$-invariant complex Finsler metric $F=\sqrt{r\phi(t,s)}$
is a weakly K\"ahler-Finsler metric if and only if $\phi(t,s)$ satisfies
\begin{equation}
(\phi-s\phi_s)[\phi+(t-s)\phi_s][\phi_s-\phi_t+s(\phi_{st}+\phi_{ss})]+s(t-s)\phi_{ss}[\phi(\phi_s-\phi_t)+s\phi_s(\phi_t+\phi_s)]=0.\label{wkc}
\end{equation}
Very recently, Cui, Guo and Zhou \cite{Zhou} obtained a special solution of \eqref{wkc} which are $U(n)$-invariant complex Rander metric.
They also gave a classification of weakly K\"ahler-Finsler metrics which are $U(n)$-invariant complex Randers metric and have constant holomorphic sectional curvatures.

All of these progress in complex Finsler geometry provide us with nontrivial examples of strongly pseudoconvex (even strongly convex) K\"ahler-Finsler metrics and weakly K\"ahler-Finsler metrics (they are nontrivial in the sense that they are non-Hermitian quadratic metrics), some of which even enjoy a very nice curvature property. Recently, there are several important progress investigating related problem in complex Finsler geometry with the assumption that the complex Finsler metric is a strongly pseudoconvex (or strongly convex) K\"ahler-Finsler metric or weakly K\"ahler-Finsler metric, we refer to \cite{cheny},\cite{li},\cite{YZ} for more details.

Therefore,  a very natural and interesting question in complex Finsler geometry one may ask is whether it is possible to generalize Theorem \ref{chen} to more general complex metric spaces, i.e., to establish  Schwarz lemmas between two strongly pseudoconvex complex Finsler manifolds. Especially when the domain manifold  is endowed with a strongly pseudoconvex K\"ahler-Finsler metric or  weakly K\"ahler-Finsler metric and the target manifold is endowed with a general strongly pseudoconvex complex Finsler metric.

In \cite{shen}, Shen and Shen obtained a Schwarz lemma from a compact complex Finsler manifold with holomorphic sectional curvature bounded from below by a negative constant into another complex Finsler manifold with holomorphic sectional curvature bounded above by a negative constant.  In the case that the domain manifold is non-compact, Wan \cite{wan}
obtained a Schwarz lemma from a complete Riemann surface with curvature bounded from below by a constant into a complex Finsler manifold with  holomorphic sectional curvature bounded from above by a negative constant. The general case, i.e.,  when the domain manifold is a complete non-compact complex  manifold endowed with a strongly pseudoconvex complex Finsler manifold, however, is still open.
It seems that the method used in \cite{wan} does not work when the domain manifold has complex dimension $\geq 2$.

As a first step towards the above question, Nie and Zhong \cite{nie} considered the case that the domain manifold is a K\"ahler manifold and the target manifold is a strongly pseudoconvex complex Finsler manifold and obtained the following Schwarz lemma (cf. Theorem 1.1 in \cite{nie}).
 \begin{theorem}\label{t-1.2}(\cite{nie})
 Suppose that $(M,ds_M^2)$ is a complete K\"ahler manifold such that its holomorphic sectional curvature is bounded from below by a constant $K_1$ and its radial sectional curvature is also bounded from below. Suppose that $(N,H)$ is a strongly pseudoconvex complex Finsler manifold whose holomorphic sectional curvature
is bounded above by a negative constant $K_2$. Let $f:M\rightarrow N$ be a holomorphic mapping. Then
\begin{equation}
(f^\ast H)(z;dz)\leq \frac{K_1}{K_2}ds_M^2.
\end{equation}
\end{theorem}
In this paper, we go a further step towards this question, and we generalize Theorem \ref{t-1.2} to the case that the domain manifold $M$  is a strongly convex weakly K\"ahler-Finsler manifold with a pole $p$.
We first establish the following theorem which relates the real Hessian of the distance function $\rho(x)$ of a real Finsler metric $G$ on a smooth manifold $M$ with the radial flag curvature of $G$.
\begin{theorem}\label{T-1.1}(cf. Theorem 3.1)
Suppose that $(M,G)$ is a real Finsler manifold with a pole $p$ such that its radial flag curvature is bounded from below by a negative constant $-K^2$. Suppose that $\gamma:[0,r] \rightarrow M$ is a normal geodesic with $\gamma(0)=p$ such that $\gamma(r)=x \neq p$. Denote the distance function from $p$ to $x$ by $\rho(x)$. Then the Hessian of $\rho$ satisfies
\begin{equation*}
H(\rho)(u,u)(x) \leq \frac{1}{\rho}+K,
\end{equation*}
where $u=u^i\frac{\partial}{\partial x^i} \in T_xM$ is a unit vector.
\end{theorem}

Using Theorem \ref{T-1.1}, we obtain the following corollary.
\begin{corollary}(cf. Corollary 3.2)
Suppose that $(M,G)$ is a real Finsler manifold with a pole $p$ such that its radial flag curvature is bounded from below by a negative constant $-K^2$. Suppose that $\gamma:[0,r] \rightarrow M$ is a normal geodesic with $\gamma(0)=p$ such that $\gamma(r)=x \neq p$. Denote the distance function from $p$ to $x$ by $\rho(x)$. Then with respect to the normal coordinates at the point $x$,
\begin{equation*}
\frac{\partial^2 \rho^2}{\partial x^i \partial x^j}u^iu^j \leq 2(2+\rho K).
\end{equation*}
\end{corollary}

 If $M$ is a complex manifold endowed with a strongly convex complex Finsler metric $G$, that is, $G$ is simultaneously a real Finsler metric on $M$ when one ignores its complex structure, then by Lemma \ref{b}and \ref{v}, we are able to establish the following theorem which gives an estimation of the Levi-form of the distance function $\rho(z)$ in terms of the radial flag curvature of a strongly convex weakly K\"ahler-Finsler metric. The following theorem plays an important role in the proof of the main theorem in this paper.
\begin{theorem}\label{T-1.4}(cf. Theorem 4.1)
Suppose that $(M,G)$ is a strongly convex weakly K\"ahler-Finsler manifold with a pole $p$ such that its radial flag curvature is bounded from below by a negative constant $-K^2$. Suppose that  $\gamma:[0,r] \rightarrow M$ is a  geodesic with $G(\dot{\gamma})\equiv 1$ such that $\gamma(0)=p$ and $\gamma(r)=z \neq p$. Denote the distance function from $p$ to $z$ by $\rho(z)$. Then
\begin{equation*}
\frac{\partial^2 \rho^2}{\partial z^\alpha\partial \bar{z}^\beta}v^\alpha\bar{v}^\beta \leq (2+\rho K),
\end{equation*}
where $v=v^\alpha\frac{\partial}{\partial z^{\alpha}}=\frac{1}{2}(u-\sqrt{-1}Ju) \in T_z^{1,0}M$ is a unit vector and $u=u^i\frac{\partial}{\partial x^i}\in T_xM$.
\end{theorem}

 Using Theorem \ref{T-1.4}, we are able to establish the main theorem in this paper as follows.
\begin{theorem} \label{T-1.3}(cf. Theorem 6.2)
Suppose that $(M,G)$ is a complete strongly convex weakly K\"ahler-Finsler manifold such that its radial flag curvature is bounded from  below and its holomorphic sectional curvature $K_G$ is bounded from below by a constant $K_1 \leq 0$. Suppose that $(N,H)$ is a strongly pseudoconvex complex Finsler manifold such that its holomorphic sectional curvature $K_H$ is bounded from above by a constant $K_2<0$.  Let $f:M\rightarrow N$ be a holomorphic mapping. Then
\begin{equation*}
(f^*H)(z;dz) \leq \frac{K_1}{K_2}G(z;dz).
\end{equation*}
\end{theorem}
\begin{remark}
If $(M,G)$ come from K\"ahler manifolds, then Theorem \ref{T-1.3} is exactly Theorem 1.1 in \cite{nie}.
\end{remark}

As an immediate application of Theorem \ref{T-1.3}, if the complex Finsler metric $H$ on  $N$  comes from a Hermitian metric, then we obtain a Schwarz lemma from a strongly convex weakly K\"ahler-Finsler manifold into a Hermitian manifold.
\begin{corollary}(cf. Corollary 7.3)
Suppose that $(M,G)$ is a complete strongly convex weakly K\"ahler-Finsler manifold such that its radial flag curvature is bounded from below and its holomorphic sectional curvature $K_G$ is bounded from below by a constant $K_1 \leq 0$. Suppose that $(N,ds^2_N)$ is a Hermitian manifold such that its holomorphic sectional curvature $K_H$ is bounded from above by a constant $K_2<0$.  Let $f:M\rightarrow N$ be a holomorphic mapping. Then
\begin{equation*}
f^*ds_N^2 \leq \frac{K_1}{K_2}G(z;dz).
\end{equation*}
\end{corollary}

As a simple application of Theorem \ref{T-1.3}, we obtain the following corollaries about the existence of non-constant holomorphic mapping between two complex Finsler manifolds.

\begin{corollary} \label{T-1.3a} (cf. Corollary 7.1)
Suppose that $(M,G)$ is a complete strongly complete convex weakly K\"ahler-Finsler manifold such that its radial flag curvature is bounded from below and its holomorphic sectional curvature is non-negative. Suppose that $(N,H)$ is a strongly pseudoconvex complex Finsler manifold such that its holomorphic sectional curvature $K_H$ is bounded from above by a constant $K_2<0$.  Then any holomorphic mapping $f$ from $M$ into $N$ is a constant.
\end{corollary}

 To illustrate Corollary \ref{T-1.3a}, we give a concrete example as follows.
\begin{example}\label{C-1.4} (cf. Corollary 7.2)
Suppose that $(\mathbb{C}^n, G)$ is a complex Minkowski space. Suppose that $(N,H)$ is a strongly pseudoconvex complex Finsler manifold such that its holomorphic sectional curvature is bounded from above by a negative constant $K_2$. Then any holomorphic mapping $f$ from $\mathbb{C}^n$ into $N$ is a constant.
\end{example}
\begin{remark}
If $(N,H)$ is a Hermitian manifold, it follows immediately that a holomorphic mapping from a complex Minkowski space into a Hermitian manifold with holomorphic sectional curvature bounded from above a negative constant is necessary a constant.
\end{remark}
\section{Preliminaries for Finsler geometry}
\setcounter{equation}{0}
\noindent

In this section, we shall recall some basic definitions and facts on real and complex Finsler geometry. We refer to \cite{abate,bao} for more details.
\subsection{Real Finsler geomtry}
\noindent

Let $M$ be a smooth manifold of real dimension $n$, and $\pi:TM \rightarrow M$ be the tangent bundle of $M$. Let $x=(x^1, \cdots,x^n)$ be a local coordinate system on an open set $U\subset M$, then a tangent vector $u$ at the point $x\in M$ can be written as
$$u=u^i \frac{\partial}{\partial x^i}\in T_xM,$$
where the Einstein's sum convention is used throughout this paper. So that we can use $(x;u)=(x^1,\cdots,x^n; u^1,\cdots,u^n)$ as local coordinate system on $\mathcal{U}=\pi^{-1}(U)\subset TM$.
We denote by $\tilde{M}:=TM\setminus\{o\}$ the complement of the zero section in $TM$. Then
$$\Big\{\partial_i:=\frac{\partial}{\partial x^i}, \dot{\partial}_i:=\frac{\partial}{\partial u^i}\Big\}_{i=1}^n$$
 gives a local frame field of the tangent bundle $T\tilde{M}$ over $\mathcal{U}$.
\begin{definition}(\cite{abate})\label{d-2.1}
A (real) Finsler metric on $M$ is a continuous function $G:TM\rightarrow [0,+ \infty)$ satisfying the following properties:

(i) $G$ is smooth on $\tilde{M}$;

(ii) $G(x;u) \geq 0$ for all $u \in T_xM$ and $x \in M$,  and $G(x;u)=0$ if and only if $u=0$;

(iii) $G(x;\lambda u)=|\lambda|^2G(x;u)$ for all $u \in T_xM$ and $\lambda \in \mathbb{R}$;

(iv) The fundamental tensor matrix $(g_{ij})$, where
$$g_{ij}:=\frac{1}{2}G_{ij}=\frac{1}{2}\frac{\partial^2 G}{\partial u^i \partial u^j}$$
is positive definite on $\tilde{M}$.
\end{definition}
A manifold $M$ endowed with a (real) Finsler metric $G$ is called a (real) Finsler manifold. Note that by definition the smoothness of $G$ is only asked on $\tilde{M}$. In fact, a real Finsler metric $G$ is smooth on the whole $TM$ if and only if it comes from a Riemannian metric \cite{abate}.

Using the projection $\pi:TM\rightarrow M$, one can define the vertical bundle as
$$\mathcal{V}=\text{ker} d\pi \subset \tilde{M}.$$
It is clear that $\{\dot{\partial}_1, \cdots, \dot{\partial}_n\}$ is a local frame for $\mathcal{V}$. The horizontal bundle is a subbundle $\mathcal{H} \subset T\tilde{M}$ such that
$$T\tilde{M}=\mathcal{H}\oplus \mathcal{V}.$$
A local frame field for $\mathcal{H}$ is given by $\{\delta_1, \cdots,\delta_n\}$, where
\begin{equation}
\delta_i= \partial_i-\Gamma_{;i}^j\partial_j, \quad i=1,\cdots,n.\label{di}
\end{equation}
Here in \eqref{di}, we have denoted
\begin{eqnarray}
\Gamma_{:i}^j&=&\frac{1}{2}G^{jk}[G_{ki;b}u^b+G_{k;i}-G_{i;k}]-\Gamma^j_{i;k}G^{kl}[G_{l;b}u^b-G_{;l}].
\end{eqnarray}

By Definition \ref{d-2.1}, one can introduce a Riemannian structure $\langle\cdot|\cdot\rangle$ on $\mathcal{V}$ by setting
$$ \forall V,W \in \mathcal{V}_u, \quad  \langle V|W\rangle_u = \frac{1}{2}G_{ij}(u)V^iW^j.$$
Let $D: \mathcal{X}(\mathcal{V}) \rightarrow \mathcal{X}(T^*\tilde{M}\otimes \mathcal{V})$ be the Cartan connection associated to $G$. Its connection $1$-forms are given by
$$\omega_i^j= \Gamma^j_{i;k}dx^k+ \Gamma_{ik}^j\psi^k,$$
where
$$\Gamma^j_{i;k}=\frac{1}{2}G^{jl}[\delta_k(G_{il})+\delta_{i}(G_{lk})-\delta_{l}(G_{ik})],\quad \Gamma_{ik}^j=\frac{1}{2}G^{jl}G_{ikl}$$
and $\psi^k=du^k+\Gamma^k_{;l}dx^l$.
Let $\chi_u: T_{\pi(u)}M \rightarrow \mathcal{H}_u$ be the horizontal lift of a vector locally defined by $\chi_u(\partial_i|_{\pi(u)})=\delta_i|_u$. Then the radial horizontal vector field $\chi: TM \rightarrow \mathcal{H}$ is defined by $\chi(u)=\chi_u(u)$.

 Using the horizontal map $\Theta: \mathcal{V} \rightarrow \mathcal{H}$ which is locally defined by $\Theta(\dot{\partial}_i)=\delta_i$, one can transfer the Riemannian structure $\langle\cdot|\cdot\rangle$ from $\mathcal{V}$ to $\mathcal{H}$ just by setting
$$ \forall H,K \in \mathcal{H}, \quad \langle H|K\rangle=\langle\Theta^{-1}(H)|\Theta^{-1}(K)\rangle.$$
 Thus one can then define a Riemannian metric, still denoted by $\langle\cdot|\cdot\rangle$ on the whole $T\tilde{M}$, just by asking for $\mathcal{H}$ to be orthogonal to $\mathcal{V}$. Using $\Theta$, one can introduce a linear connection (still denoted by $D$) on $\mathcal{H}$ by setting
$$\forall H \in \mathcal{X}(\mathcal{H}), \quad DH=\Theta(D(\Theta^{-1}(H))).$$
Thus one can extend and obtain a good linear connection on $\mathcal{V}$ to a linear connection on $T\tilde{M}$. It is then easy to check that
$$X\langle Y| Z\rangle = \langle D_XY|Z \rangle + \langle Y| D_XZ \rangle$$
for all $X, Y, Z \in \mathcal{X}(T\tilde{M})$.

Let $\nabla: \mathcal{X}(T\tilde{M})\times \mathcal{X}(T\tilde{M}) \rightarrow \mathcal{X}(T\tilde{M})$ be the covariant differentiation that associated to the Cartan connection $D$. Its curvature $\Omega$ is given by
$$\Omega(X,Y)Z=\nabla_X\nabla_YZ-\nabla_Y\nabla_XZ-\nabla_{[X,Y]}Z$$
for all $X,Y,Z \in \mathcal{X}(T\tilde{M})$.

 Now we are in a position to introduce the flag curvature in real Finsler geometry. It is a natural generalization of the sectional curvature in Riemannian geometry. Given $x \in M$, a flag in the tangent space $T_xM$ is a pair $(P,u)$, where $P$ is a two-dimensional subspace (tangent plane) of $T_xM$ such that $0 \neq u \in P$ and $P=\mbox{span}\{u,X\}$. The flag curvature $K^G(P,u)$ is given by
\begin{equation*}\begin{split}
K^G(P,u)=K^G(u,X)&=\frac{\langle \Omega(X^H,\chi(u))\chi(u)|X^H\rangle_u}{\langle \chi(u)|\chi(u)\rangle_u \langle X^H|X^H \rangle_u-\langle \chi(u)|X^H\rangle_u^2}\\
&=\frac{\langle \Omega(\chi(u),X^H)X^H|\chi(u)\rangle_u}{\langle \chi(u)|\chi(u)\rangle_u \langle X^H|X^H \rangle_u-\langle \chi(u)|X^H\rangle_u^2},
\end{split}\end{equation*}
where $X^H$ is the horizontal lifting of $X\in T_xM$ and $\langle\cdot|\cdot\rangle_u$ is the Riemannian metric on $\mathcal{H}_u$ induced by the fundamental metric tensor $(g_{ij})$ of $G$.

Let $\gamma:[0,r] \rightarrow M$ be a regular curve, and $\xi$ a vector field along $\gamma$. We say that $\xi$ is parallel along $\gamma$ if $\nabla_{T^H}\xi^H=0$, where $T=\dot{\gamma}$. The length of the regular curve $\gamma$ with respect to the Finsler metric is given by
$$L(\gamma)=\int_0^r (G(\dot{\gamma}(t)))^{\frac{1}{2}}dt.$$
If $\Sigma(s,t):(-\varepsilon,\varepsilon)\times [0,r] \rightarrow M$ is a regular variation of $\gamma$, we can define the function $l_{\Sigma}:(-\varepsilon,\varepsilon) \rightarrow [0,+\infty)$ by $l_{\Sigma}(s)=L(\gamma_s)$, where $\gamma_s(t)=\Sigma(s,t)$. By Corollary 1.5.2 in \cite{abate}, we know that $\gamma$ is a geodesic iff $\frac{dl_{\Sigma}}{ds}(0)=0$, that is $\nabla_{T^H}{T^H}$=0, where $T=\dot{\gamma}$.
 \begin{theorem}(\cite{abate})\label{l}
Let $G: TM \rightarrow [0,+\infty)$ be a Finsler metric on a manifold $M$. Take a geodesic $\gamma_0:[0,r]\rightarrow M$, with $G(\dot{\gamma}_0)\equiv 1$, and let $\Sigma: (-\varepsilon,\varepsilon)\times [0,r] \rightarrow M$ be a regular variation of $\gamma_0$. Then
\begin{equation*}\begin{split}
\frac{d^2l_{\Sigma}}{ds^2}(0)&=\langle\nabla_{U^H}U^H|T^H\rangle_{\dot{\gamma}_0}|_0^r\\
&+\int^r_0\Big[\|\nabla_{T^H}U^H\|^2_{\dot{\gamma}_0}-\langle\Omega(T^H,U^H)U^H|T^H\rangle_{\dot{\gamma}_0}-\Big|\frac{\partial}{\partial t}\langle U^H|T^H\rangle_{\dot{\gamma}_0}\Big|^2\Big]dt,
\end{split}\end{equation*}
where $\|H\|^2=\langle H|H\rangle_u$ for all $u \in \tilde{M}$, $H \in \mathcal{H}_u$ and $T=\dot{\gamma}, U= \Sigma_*(\frac{\partial}{\partial s})|_{s=0}$. In particular, if the variation $\Sigma$ is fixed we have
\begin{equation*}
\frac{d^2l_{\Sigma}}{ds^2}(0)=
\int^r_0\Big[\|\nabla_{T^H}U^H\|^2_{\dot{\gamma}_0}-\langle\Omega(T^H,U^H)U^H|T^H\rangle_{\dot{\gamma}_0}-\Big|\frac{\partial}{\partial t}\langle U^H|T^H\rangle_{\dot{\gamma}_0}\Big|^2\Big]dt.
\end{equation*}
\end{theorem}
\begin{remark}
By Lemma 4.1.1 in \cite{shen1} and the linearly parallel translations are the same with the parallel translations for a real Finsler manifold, $\langle U^H|T^H\rangle_{\dot{\gamma}_0}=g_{\dot{\gamma}_0}(U,T)$ is a constant. Therefore, we have
\begin{equation*}
\frac{d^2l_{\Sigma}}{ds^2}(0)=
\int^r_0[\|\nabla_{T^H}U^H\|^2_{\dot{\gamma}_0}-\langle\Omega(T^H,U^H)U^H|T^H\rangle_{\dot{\gamma}_0}]dt.
\end{equation*}
\end{remark}

A vector field $J$ along $\gamma$ is called a Jacobi field if it satisfies the following equation:
\begin{equation*}
\nabla_{T^H}\nabla_{T^H}J^H - \Omega(T^H,U^H)T^H \equiv 0,
\end{equation*}
where $T=\dot{\gamma}$, $T^H(\dot{\gamma})=\chi(\dot{\gamma})$ and $J^H(t)=\chi_{\dot{\gamma}}(J(t))$. The set of all Jacobi fields along $\gamma$ will be denoted by $\mathcal{J}(\gamma)$. A proper Jacobi field is a $J \in \mathcal{J}(\gamma)$ such that
\begin{equation*}
\langle J^H|T^H\rangle_T \equiv 0.
\end{equation*}
Denote by $\mathcal{J}_0(\gamma)$ the set of all proper Jacobi fields along $\gamma$.

Let $\gamma:[0,r] \rightarrow M$ be a normal geodesic in a real Finsler manifold $(M,G)$; we shall denote by $\mathcal{X}[0,r]$ the space of all piecewise smooth vector fields $\xi$ along $\gamma$ such that
\begin{equation*}
\langle\xi^H|T^H\rangle_T \equiv 0,
\end{equation*}
where $T=\dot{\gamma}$. Moreover, we shall denote by $\mathcal{X}_0[0,r]$ the subspace of all $\xi \in \mathcal{X}[0,r]$ such that $\xi(0)=\xi(r)=0$.
\begin{definition}(\cite{abate})\label{I}
The Morse index form $I=I^r_0:\mathcal{X}[0,r]\times\mathcal{X}[0,r] \rightarrow \mathbb{R}$ of the normal geodesic $\gamma:[0,r] \rightarrow M$ is the symmetric bilinear form
\begin{equation*}
I(\xi,\eta)= \int_0^r[\langle\nabla_{T^H}\xi^H|\nabla_{T^H}\eta^H\rangle_T-\langle\Omega(T^H,\xi^H)\eta^H|T^H\rangle_T]dt
\end{equation*}
for all $\xi,\eta \in \mathcal{X}[0,r]$, where $T=\dot{\gamma}$.
\end{definition}
\begin{remark}
Note that, if $\langle\nabla_{U^H}U^H(0)|T^H(0)\rangle=\langle\nabla_{U^H}U^H(r)|T^H(r)\rangle=0$ and $U \in \mathcal{X}[0,r]$ is the transversal vector of a regular variation $\Sigma$ of $\gamma$, then by Theorem \ref{l} we have
\begin{equation*}
I(U,U)=\frac{d^2l_{\Sigma}}{ds^2}(0).
\end{equation*}
\end{remark}
\subsection{Complex Finsler geometry}
\noindent

Let $M$ be a complex manifold of complex dimension $n$. Let $\{z^1,\cdots,z^n\}$ be a set of local complex coordinates, and let $\{\frac{\partial}{\partial z^{\alpha}}\}_{1 \leq \alpha \leq n}$ be the corresponding natural frame of $T^{1,0}M$. So that any non-zero element in $\tilde{M}=T^{1,0}M \setminus \{\text{zero section}\}$ can be written as
$$v=v^{\alpha}\frac{\partial}{\partial z^{\alpha}} \in \tilde{M},$$
where we also use the Einstein's sum convention. In this way, one gets a local coordinate system on the complex manifold $\tilde{M}$:
$$(z;v)=(z^1,\cdots,z^n;v^1,\cdots,v^n).$$

In this paper, we still denote a complex Finsler metric by $G$, depending on the  actual situation.
\begin{definition}(\cite{abate},\cite{kobayashi})\label{D-2.3}
A complex Finsler metric $G$ on a complex manifold $M$ is a continuous function $G:T^{1,0}M \rightarrow [0,+ \infty)$ satisfying

(i) $G$ is smooth on $\tilde{M}:=T^{1,0}M\setminus\{\mbox{zero section}\}$;

(ii) $G(z;v) \geq 0$ for all $v \in T_z^{1,0}M$ with $z \in M$, and $G(z;v)=0$ if and only if $v=0$;

(iii) $G(z;\zeta v)=|\zeta|^2G(z;v)$ for all $v \in T^{1,0}_zM$ and $\zeta \in \mathbb{C}$.

\end{definition}
\begin{definition}(\cite{abate},\cite{kobayashi}) \label{D-4.1}
A complex Finsler metric $G$ is called strongly pseudoconvex if the Levi matrix
\begin{equation*}
(G_{\alpha \overline{\beta}})= \Big(\frac{\partial^2 G}{\partial v^{\alpha}\partial \overline{v}^{\beta}}\Big)
\end{equation*}
is positive definite on $\tilde{M}$.
\end{definition}
\begin{remark}[\cite{abate}]
Any $C^\infty$ Hermitian metric on a complex manifold $M$ is naturally a strongly pseudoconvex complex Finsler metric. Conversely,
if a complex Finsler metric $G$ on a complex manifold $M$ is $C^\infty$ over the whole holomorphic tangent bundle $T^{1,0}M$, then it is necessary a $C^\infty$ Hermitian metric. That is, for any $(z;v)\in T^{1,0}M$
$$
G(z;v)=g_{\alpha\overline{\beta}}(z)v^\alpha\overline{v}^\beta
$$
for a $C^\infty$ Hermitian tensor $g_{\alpha\overline{\beta}}$ on $M$.
For this reason, in general the non-trivial (non-Hermitian quadratic) examples of complex Finsler metrics are only required to be smooth over the slit holomorphic tangent bundle $\tilde{M}$.
\end{remark}

 In the following, we follow the notions in Abate and Patrizio \cite{abate}. We shall denote   by indexes like $\alpha, \overline{\beta}$ and so on the derivatives with respect to the $v$-coordinates; for instance,
\begin{equation*}
G_{\alpha \overline{\beta}}= \frac{\partial^2 G}{\partial v^{\alpha}\partial \overline{v}^{\beta}}.
\end{equation*}
On the other hand, the derivatives with respect to the $z$-coordinates will be denoted by indexes after a semicolon; for instance,
\begin{equation*}
G_{;\mu \overline{\nu}}= \frac{\partial^2 G}{\partial z^{\mu}\partial \overline{z}^{\nu}}\quad \text{or}\quad G_{\alpha;\overline{\nu}}= \frac{\partial^2 G}{\partial \overline{z}^{\nu}\partial v^\alpha}.
\end{equation*}

 Let $G:T^{1,0}M \rightarrow [0, +\infty)$ be a complex Finsler metric on a complex manifold $M$. To $G$, we may associate a function $G^\circ: TM \rightarrow [0, +\infty)$ just by setting
$$G^\circ(u)=G(u_\circ), \quad \forall u \in TM,$$
where $_\circ:TM \rightarrow T^{1,0}M$ is an $\mathbb{R}$-isomorphism given by
$$u_\circ=\frac{1}{2}(u-iJu), \quad  \forall  u \in TM,$$
where $J$ is the canonical complex structure on $M$.\\
Correspondingly, the inverse $^\circ:T^{1,0}M \rightarrow TM$ is given by
$$v^{\circ}=v+\bar{v}, \quad \forall v \in T^{1,0}M,$$
where $\bar{v}$ denotes the complex conjugation of $v$.
\begin{definition}(\cite{abate})
A complex Finsler metric $G$ is called strongly convex if $G^\circ$ is a real Finsler metric on $M$ (considered as a smooth manifold of real dimension $2n$).
\end{definition}
Using the projective map  $\pi:T^{1,0}M\rightarrow M$, which is a holomorphic mapping, one can define the holomorphic vertical bundle
\begin{equation*}
\mathcal{V}^{1,0}:=\ker{d\pi} \subset T^{1,0}\tilde{M}.
\end{equation*}
It is obvious that $\{\frac{\partial}{\partial v^1},\cdots,\frac{\partial}{\partial v^n}\}$ is a local frame for $\mathcal{V}^{1,0}$.

The complex horizontal bundle of type $(1,0)$ is a complex subbundle $\mathcal{H}^{1,0} \subset T^{1,0}\tilde{M}$ such that
\begin{equation*}
T^{1,0}\tilde{M}=\mathcal{H}^{1,0} \oplus \mathcal{V}^{1,0}.
\end{equation*}
Note that $\{\delta_1,\cdots,\delta_n\}$ is a local frame for $\mathcal{H}^{1,0}$, where
\begin{equation*}
\delta_{\alpha}= \partial_\alpha-\Gamma_{;\alpha}^{\beta}\dot{\partial}_\beta, \quad \Gamma_{;\alpha}^{\beta}:= G^{\beta \overline{\gamma}}G_{\overline{\gamma};\alpha}.
\end{equation*}
Here and in the following, we write $\partial_\alpha:=\frac{\partial}{\partial z^\alpha}$ and $\dot{\partial}_\beta:=\frac{\partial}{\partial v^\beta}$.

For a holomorphic vector bundle whose fiber metric is a Hermitian metric, there is naturally associated a unique complex linear connection (the Chern connection or Hermitian connection) with respect to which the metric tensor is parallel. Since each strongly pseudoconvex complex Finsler metric $G$ on a complex manifold $M$ naturally induces a Hermitian metric on the holomorphic vertical bundle $\mathcal{V}^{1,0}$. It follows that there exists a unique good complex vertical connection $D:\mathcal{X}(\mathcal{V}^{1,0})\rightarrow \mathcal{X}(T_{\mathbb{C}}^\ast\tilde{M}\otimes \mathcal{V}^{1,0})$ compatible with the Hermitian structure in $\mathcal{V}^{1,0}$. This connection is called the Chern-Finsler connecction (see \cite{abate}). The connection $1$-form are given by
\begin{equation*}
\omega_{\beta}^{\alpha}:=G^{\alpha \overline{\gamma}}\partial G_{\beta \overline{\gamma}}= \Gamma_{\beta;\mu}^{\alpha}dz^{\mu}+\Gamma_{\beta \gamma}^{\alpha} \psi^{\gamma},
\end{equation*}
where
\begin{equation*}
 \Gamma_{\beta;\mu}^{\alpha} =G^{\overline{\tau}\alpha}\delta_{\mu}(G_{\beta \overline{\tau}}),\quad\Gamma_{\beta \gamma}^{\alpha}=G^{\overline{\tau}\alpha}G_{\beta \overline{\tau} \gamma},\quad \psi^{\alpha}= dv^{\alpha}+\Gamma_{;\mu}^{\alpha}dz^{\mu}.
\end{equation*}

 Let
 $$\theta^{\alpha}=\frac{1}{2}[\Gamma_{\nu;\mu}^{\alpha}-\Gamma_{\mu;\nu}^{\alpha}]dz^{\mu}\wedge dz^{\nu}+\Gamma^{\alpha}_{\nu\gamma}\psi^{\gamma}\wedge dz^{\nu}$$
 be the local expression of the $(2,0)$-torsion $\theta=\theta^{\alpha}\delta_{\alpha}$ for the Chern-Finsler connection.
 \begin{definition}(\cite{abate})
 In local coordinates, a complex Finsler metric $G$ is called strongly K\"ahler if $\Gamma^{\alpha}_{\nu;\mu}=\Gamma_{\mu;\nu}^{\alpha}$; it is called K\"ahler if $[\Gamma_{\nu;\mu}^{\alpha}-\Gamma_{\mu;\nu}^{\alpha}]v^{\mu}=0$; it is called weakly K\"ahler if $G_{\alpha} [\Gamma_{\nu;\mu}^{\alpha}-\Gamma_{\mu;\nu}^{\alpha}]v^{\mu}=0$.
 \end{definition}
 \begin{remark}
In \cite{chenb}, Chen and Shen  proved that a complex Finsler metric $G$ is a K\"ahler-Finsler metric  iff it is a strongly K\"ahler-Finsler metric, thus leaving two notions of K\"ahlerian in complex Finsler setting.
 We also point it out here that the notions of K\"ahler-Finsler metric and weakly K\"ahler-Finsler metric are not equivalent since there are indeed examples which are weakly K\"ahler-Finsler metrics but not K\"ahler-Finsler metrics  \cite{Zhou}.
 \end{remark}

 In this paper, we only consider strongly pseudoconvex complex Finsler metrics on a complex manifold $M$. The curvature form of the Chern-Finsler connection can be expressed as
\begin{equation*}
\Omega^{\alpha}_{\beta}:=R^{\alpha}_{\beta;\mu\overline{\nu}}dz^{\mu}\wedge d\overline{z}^{\nu}+R^{\alpha}_{\beta \delta;\overline{\nu}}\psi^{\delta} \wedge d\overline{z}^{\nu}+R^{\alpha}_{\beta \overline{\gamma};\mu}dz^{\mu}\wedge \overline{\psi^{\gamma}}+
R^{\alpha}_{\beta \delta \overline{\gamma}}\psi^{\delta}\wedge \overline{\psi^{\gamma}},
\end{equation*}
where
\begin{eqnarray*}
R^{\alpha}_{\beta;\mu\overline{\nu}} &=& -\delta_{\overline{\nu}}(\Gamma^{\alpha}_{\beta;\mu})-\Gamma^{\alpha}_{\beta\sigma}\delta_{\overline{\nu}}(\Gamma^{\sigma}_{;\mu}),\\
R^{\alpha}_{\beta\delta;\overline{\nu}} &=& -\delta_{\overline{\nu}}(\Gamma^{\alpha}_{\beta\delta}),\\
R^{\alpha}_{\beta\overline{\gamma};\mu}&=&-\dot{\partial}_{\overline{\gamma}}(\Gamma^{\alpha}_{\beta;\mu})-\Gamma^{\alpha}_{\beta\sigma}\Gamma^{\sigma}_{\overline{\gamma};\mu},\\
R^{\alpha}_{\beta\delta\overline{\gamma}}&=&-\dot{\partial}_{\overline{\gamma}}(\Gamma^{\alpha}_{\beta\delta}).
\end{eqnarray*}

\begin{definition}(\cite{abate})
Let $\mu =g d\zeta \otimes d \zeta$ be a Hermitian metric defined in a neighborhood of the origin in $\mathbb{C}$. Then the Gaussian curvature $K(\mu)(0)$ of $\mu$ at the origin is given by
\begin{equation*}
K(\mu)(0) = -\frac{1}{2g(0)}(\Delta \log g)(0),
\end{equation*}
where $\Delta$ denotes the usual Laplacian
\begin{equation*}
\Delta u = 4 \frac{\partial^2 u}{\partial \zeta \partial \bar{\zeta}}.
\end{equation*}
\end{definition}
For a strongly pseudoconvex complex Finsler metric $G$ on $M$,  one can also introduce the notion of holomorphic sectional curvature.
\begin{definition}(\cite{abate})
Let$(M,G)$ be a strongly pseudoconvex complex Finsler metric on a complex manifold $M$, and take $v \in \tilde{M}$. Then the holomorphic sectional curvature $K_G(v)$ of $G$ along $v$ is given by
\begin{equation*}
K_G(v)=K_G(\chi(v))=\frac{2}{G(v)^2}\langle \Omega(\chi,\bar{\chi})\chi,\chi\rangle_v,
\end{equation*}
where $\chi=v^{\alpha}\delta_{\alpha}$ is the complex radial horizontal vector field and $\Omega$ is the curvature tensor of the Chern-Finsler connection associated to $(M,G)$.
\end{definition}
In complex Finsler geometry, Abate and Patrizio\cite{abate} (see also Wong and Wu\cite{wong}) proved that the holomorphic sectional curvature of $G$ at a point $z \in M$ along a tangent direction $v \in T_z^{1,0}M$ is the maximum of the Gaussian curvatures of the induced Hermitian metrics among all complex curves in $M$ which pass through $z$ and tangent at $z$ in the direction $v$.
\begin{lemma}(\cite{abate,wong})\label{L-2.1}
Let $(M,G)$ be a complex Finsler manifold, $v \in T^{1,0}_zM$ be a nonzero tangent vector tangent at a point $z \in M$. Let $\mathcal{C}$ be the set of complex curves in $M$ passing through $z$ which are tangent to $v$ at $z$. Then the holomorphic sectional curvature $K_G(v)$ of $G$ satisfies the condition
\begin{equation*}
K_G(v)=\max_{S \in \mathcal{C}}K(S)(z),
\end{equation*}
where $K(S)$ is the Gaussian curvature of the complex curve $S$ with the induced metric.
\end{lemma}

\section{The estimation of the distance function on real Finsler manifolds}
\setcounter{equation}{0}
\noindent

In this section, we follow the notations in \cite{abate}, \cite{shen2}. We first introduce the definition of Hessian in real Finsler geometry. Then we obtain an estimation of the distance function $\rho(x)$ of a real Finsler metric $G$ and an equality which establishes a relationship between the Hessian of the distance function associated to $G$ and the Morse index form on a real Finsler manifold $(M,G)$, see Proposition \ref{W}. Basing on this, we obtain an inequality which relates the real Hessian of the distance function $\rho(x)$ and the radial flag curvature of the real Finsler metric $G$, see Theorem \ref{a}.

Now, we introduce the definition of Hessian in real Finsler geometry. The main ideas come from Shen (see \cite{shen2}) and the corresponding definition in Riemannian case. Let $(M,G)$ be a real Finsler manifold and let $f$ be a smooth function on $M$. The Legendre transform $\ell: TM \rightarrow T^*M$ is defined by
\begin{equation*}
\ell(Y)=
\begin{cases}
\langle Y^H|\,{\cdot}^H\rangle_Y,&  \text{if} \quad Y \neq 0;\\
0, & \text{if} \quad Y = 0.
\end{cases}
\end{equation*}
Then the gradient of $f$ is defined by $\hat{\nabla} f=\ell^{-1}(df)$ (see \cite{shen2}). Therefore we have
\begin{equation*}
df(X)=\langle (\hat{\nabla} f)^H| X^H\rangle_{\hat{\nabla} f}.
\end{equation*}
By the properties of the Legendre transform $\ell$, it is easy to check that
$$\hat{\nabla} \rho^2=\ell^{-1}(d\rho^2)=\ell^{-1}(2\rho d\rho)=2\rho \ell^{-1}(d\rho)=2\rho \hat{\nabla} \rho.$$
In local coordinates, $\hat{\nabla} f$ can be expressed by
\begin{equation*}
\hat{\nabla} f = g^{ij}(\hat{\nabla} f) \frac{\partial f}{ \partial x^j}\frac{\partial }{\partial x^i}.
\end{equation*}
Then, in $M_f:=\{x \in M | df(x) \neq 0\}$, the Hessian of $f$ is defined as (see \cite{wub}):
\begin{equation*}
H(f)(X,Y)=D^2f(X,Y)=X(Yf)-(\nabla_{X^H}Y^H) f|_{\hat{\nabla}f}, \forall X,Y \in TM|_{M_f},
\end{equation*}
where $D$ is the Cartan connection of $G$.
It is easy to see that $H(f)$ is symmetric, and it can be written as
\begin{equation*}
H(f)(X,Y)=\langle \nabla_{X^H}(\hat{\nabla} f)^H| Y^H\rangle|_{\hat{\nabla} f}.
\end{equation*}
\begin{definition}(\cite{abate},\cite{li})
$(M,G)$ is called a real Finsler manifold with a pole $p$ if the exponential map $\exp_p:T_pM \rightarrow M$ is an $E$-diffeomorphism at $p$.
 \end{definition}
 \begin{remark}
 If $(M,G)$ is a real Berwald manifold, we know that $\exp_p:T_pM \rightarrow M$ is a smooth map (see \cite{bao}). Therefore we say $(M,G)$ is a real Berwald manifold with a pole $p$ if the exponential map $\exp_p:T_pM \rightarrow M$ is an diffeomorphism at $p$.
\end{remark}
Given a real Finsler manifold $M$ with a  pole $p$, the radial vector field is the unit vector field $T$ defined on $M-\{p\}$, such that for any $x\in M-\{p\}$, $T(x)$ is the unit vector tangent to the unique geodesic joining $p$ to $x$ and pointing away from $p$. A plane $\pi$ in $T_xM$ is called a radial plane if $\pi$ contains $T(x)$. By the radial flag curvature of a real Finsler manifold $(M,G)$, we mean the restriction of the flag curvature function to all the radial planes, we refer to \cite{li} for more details.
Note that if $M$ possess a pole $p$, then it is complete. In this case, we denote the distance function from $p$ to $x$ by $\rho(x)$. By Proposition 6.4.2 in \cite{bao}, we know that that $\rho^2(x)$ is smooth on $M-\{p\}$. The following proposition was actually outlined and used in Li and Qiu \cite{li}. We single it out and give a brief proof here since we need it to prove Theorem \ref{a}.
\begin{proposition}\label{W}
Let $(M,G)$ be a real Finsler manifold with a pole $p$. Let $\gamma:[0,r] \rightarrow M$ be a normal geodesic with $\gamma(0)=p$ such that $\gamma(r)=x(\neq p)$. Then
\begin{equation*}
H(\rho)(X,X)= \int_0^r[\langle\nabla_{T^H}J^H|\nabla_{T^H}J^H\rangle-\langle\Omega(T^H,J^H)J^H|T^H\rangle]_Tdt=I(J,J),
\end{equation*}
where $J$ is a Jacobi field along the geodesic $\gamma$ such that $J(0)=0$, $J(r)=X$, $T=\dot{\gamma}$.
\end{proposition}
\begin{proof}
Let $T(x)$ be the unit vector tangent to the unique geodesic joint $p$ to $x$ and pointing away from $p$. Consider the local $g_T$-orthogonal decomposition near $x$ (see pp. 186 in \cite{wub}),
\begin{equation*}
T_xM=\text{span}\{ T(x)\oplus T^{\bot}(x)\},
\end{equation*}
where
 $$T^{\bot}(x)=\{X\in T_xM|\langle T^H(x)|X^H\rangle=0\}.$$
We assert that there are also orthogonal decompositions relative to $H(\rho)$, in the sense that
$$H(\rho)(T(x),T^{\bot}(x))=0.$$
To show this, let $X \in T^{\bot}(x)$, then
$$H(\rho)(T(x),X)=H(\rho)(X,T(x))=XT(\rho)(x)-\nabla_{X^H}T^H \rho|_{\hat{\nabla} \rho}(x).$$
Since $\hat{\nabla} \rho =T$ and $T(\rho)=F(\hat{\nabla} \rho)=1$, this implies
\begin{eqnarray*}
H(\rho)(T(x),X)&=&-\nabla_{X^H}T^H\rho|_{\hat{\nabla} \rho}(x)\\
&=&-\langle(\hat{\nabla} \rho)^H, \nabla_{X^H}T^H\rangle_{\hat{\nabla} \rho}\\
&=&-\langle T^H,\nabla_{X^H}T^H\rangle_{\hat{\nabla}\rho}\\
&=&-\frac{1}{2}X^H\langle T^H|T^H\rangle\\
&=&0.
\end{eqnarray*}

By the definition of Hessian, it is clear that
\begin{equation*}
H(\rho)(T(x),T(x))=0.
\end{equation*}

Let $X\in T^{\bot}(x)$ and $\zeta:(-\varepsilon,\varepsilon) \rightarrow M$ be a normal geodesics such that $\zeta(0)=x$ and $\dot{\zeta}(0)=X$. Let $\gamma_s:[0,r] \rightarrow M$ be a variation of $\gamma$ such that $\gamma_s$ is the unique geodesic joining $p$ to $\zeta(s)$. Note that:

(i) the transversal vector field $J=\frac{d}{ds}(\gamma_s(t))|_{s=0}$ of $\gamma_s$ along $\gamma$ is a Jacobi field;

(ii) $J(0)=0$ and $J(r)=X$.

(iii) $\langle J^H|\dot{\gamma}^H\rangle_{\dot{\gamma}}=0$ (see Corollary 1.7.5 in \cite{abate}).

Therefore by Definition \ref{I} and Theorem \ref{l}, we have
\begin{equation*}\begin{split}
H(\rho)(X,X)&=X(X(\rho))-\nabla_{X^H}X^H(\rho)|_{\hat{\nabla} \rho(x)}\\
&=X(\dot{\zeta}(\rho))|_{s=0}-\nabla_{X^H}{\dot{\zeta}^H}(\rho)\Big|_{\hat{\nabla} \rho(x)}\\
&=\dot{\zeta}(\dot{\zeta}(\rho))|_{s=0}-\langle\nabla_{\dot{\zeta}^H}\dot{\zeta}^H|(\hat{\nabla}\rho)^H\rangle_{\hat{\nabla}\rho(x)}\\
&=\frac{d^2l_{\Sigma}}{ds^2}(0)-\langle\nabla_{\dot{\zeta}^H}{\dot{\zeta}^H}|(\hat{\nabla} \rho)^H\rangle_{\hat{\nabla} \rho(x)}\\
&=I(J,J).
\end{split}\end{equation*}
\end{proof}
\begin{remark}
 If $(M,G)$ is a complete real Finsler manifold and there are no cut points on $M$, then the assumption of a pole $p$ is not need.
\end{remark}
\begin{proposition}(\cite{abate})\label{c}
Let $\gamma:[0,r] \rightarrow M$ be a normal geodesic on a Finsler manifold $(M,G)$ which contains no conjugate points. Let $\eta \in \mathcal{X}[0,r]$, and let $J$ be a Jacobi field along $\gamma$ such that $J(0)=\eta(0)$ and $J(r)=\eta(r)$. Then
\begin{equation*}
I(J,J) \leq I(\eta,\eta),
\end{equation*}
with equality holds iff $\eta \equiv J$.
\end{proposition}

The following theorem gives an estimation of the Hessian of the distance function $\rho(x)$ in terms of the upper bound of the radial flag curvature of a real Finsler manifold $(M,G)$.
\begin{theorem}\label{a}
Suppose that  $(M,G)$ is a real Finsler manifold with a pole $p$ such that its radial flag curvature is bounded from below by a negative constant $-K^2$. Suppose that $\gamma:[0,r] \rightarrow M$ is a normal geodesic with $\gamma(0)=p$ such that $\gamma(r)=x(\neq p)$. Then
\begin{equation*}
H(\rho)(u,u)(x) \leq \frac{1}{\rho}+K,
\end{equation*}
where $u=u^i\frac{\partial}{\partial x^i} \in T_xM$ is a unit vector.
\end{theorem}
\begin{proof}
Firstly, let $u \in T_xM$ be a unit vector.
By Proposition \ref{W}, we have
\begin{equation}\label{A}
H(\rho)(u,u)= \int_0^r[\langle\nabla_{T^H}J^H|\nabla_{T^H}J^H\rangle_T-\langle\Omega(T^H,J^H)J^H|T^H\rangle_T]dt=I(J,J),
\end{equation}
where $J$ is a Jacobi field along the geodesic $\gamma$ such that $J(0)=0$ and $J(r)=u$.

 Let $\eta(t)$ be a unit vector field along $\gamma$  such that $\eta(r)=u$ and $\nabla_{T^H}(\eta(t))^H=0$. Set $\xi(t)=(\frac{t}{r})^{\alpha}\eta(t)$ and $\alpha >1$, then it is clear that $\xi(0)=J(0)=0$ and $\xi(r)=J(r)=u$.
By Proposition \ref{c} and Definition \ref{I}, we have
\begin{equation}\begin{split}\label{z}
I(J,J) &\leq \int_0^r[\langle\nabla_{T^H}\xi^H|\nabla_{T^H}\xi^H\rangle_T-\langle\Omega(T^H,\xi^H)\xi^H|T^H\rangle_T]dt\\
& \leq \int_0^r\Big[\Big\langle{\alpha}\Big(\frac{t}{r}\Big)^{\alpha-1}\xi^H\Big|{\alpha}\Big(\frac{t}{r}\Big)^{\alpha-1}\xi^H\Big\rangle_T +K^2 \langle T^H| T^H\rangle_T\langle \xi^H|\xi^H\rangle_T\Big]dt\\
&= \int_0^r\Big[\alpha^2\Big(\frac{t}{r}\Big)^{2(\alpha-1)}+K^2\Big(\frac{t}{r}\Big)^{2\alpha}\Big]dt\\
& \leq \frac{1}{r}+\frac{(\alpha-1)^2}{2\alpha-1}\frac{1}{r}+\frac{K^2r}{(2\alpha+1)},
\end{split}\end{equation}
where in the second inequality we used the condition that $-K^2$ is the lower bound of the radial flag curvature.

We can take a suitable $\alpha$, such that
\begin{equation*}
\frac{(\alpha-1)^2}{2\alpha-1}\frac{1}{r}=\frac{K^2r}{(2\alpha+1)}.
\end{equation*}
Therefore, we obtain
\begin{equation}\label{z-1}
\frac{(\alpha-1)^2}{2\alpha-1}\frac{1}{r}+\frac{K^2r}{(2\alpha+1)} =2 \sqrt{\frac{K^2(\alpha-1)^2}{(4\alpha^2-1)}}= \sqrt{\frac{K^2(4\alpha^2-8\alpha+4)}{(4\alpha^2-1)}} \leq K,
\end{equation}
where in the last step we used the fact that $ \alpha>1$.\\
By$(\ref{z})$, $(\ref{z-1})$ and the equality $\rho=r$, we get
\begin{equation}\label{EQ-3.4}
I(J,J) \leq \frac{1}{\rho}+K\Big.
\end{equation}
Plugging \eqref{EQ-3.4} into \eqref{A}, we obtain
$$H(\rho)(u,u)(x) \leq \frac{1}{\rho}+K.$$
This completes the proof of Theorem \ref{a}.
\end{proof}

As a simple application of Theorem \ref{a}, we obtain the following result.
\begin{corollary}\label{C-3.1}
Suppose that $(M,G)$ is a real Finsler manifold with a pole $p$ such that its radial flag curvature is bounded from below by a negative constant $-K^2$. Suppose that $\gamma: [0,r] \rightarrow M$ is a normal geodesic with $\gamma(0)=p$ such that $\gamma(r)=x \neq p$. Then with respect to the normal coordinates at the point $x$,
\begin{equation*}
\frac{\partial ^2 \rho}{\partial x^i \partial x^j}(x) \leq \Big(\frac{1}{\rho}+K\Big)\delta_{ij}
\end{equation*}
where $\delta_{ij}$ is the Kronecker symbol.
\end{corollary}
\begin{proof} For any given ponit $(x_0,(\hat{\nabla} \rho)(x_0)) \in \tilde{M}$, there exists a local coordinate system $(x^1,\cdots,x^n,u^1,\cdots,u^n)$ in a neighborhood of $(x_0,(\hat{\nabla} \rho)(x_0))$ such that $G_{ij}(x_0,(\hat{\nabla} \rho)(x_0))=\delta_{ij}$ and $\Gamma_{i;j}^k(x_0,(\hat{\nabla} \rho)(x_0))=0$ for $i,j,k=1,\cdots,n$. By the definition of Hessian, we have
\begin{equation*}\begin{split}
H(\rho)\Big(u^i\frac{\partial}{\partial x^i},u^j\frac{\partial}{\partial x^j}\Big)&=u^iu^jH(\rho)\Big(\frac{\partial}{\partial x^i},\frac{\partial}{\partial x^j}\Big)\\
&=u^iu^j\Big[\frac{\partial ^2 \rho}{\partial x^i \partial x^j}+\Gamma_{i;j}^k(x_0;\hat{\nabla} \rho(x_0))\frac{\partial\rho}{\partial x^k}\Big],
\end{split}\end{equation*}
Thus at the point $(x_0,(\hat{\nabla}\rho)(x_0))$, we have
\begin{equation*}
H(\rho)\Big(u^i\frac{\partial}{\partial x^i},u^j\frac{\partial}{\partial x^j}\Big)(x_0)=\Big(u^iu^j\frac{\partial ^2 \rho}{\partial x^i \partial x^j}\Big)(x_0).
\end{equation*}
By Theorem \ref{a}, we have
\begin{equation*}
\frac{\partial ^2 \rho}{\partial x^i \partial x^j}(x_0) \leq \Big(\frac{1}{\rho}+K\Big) \delta_{ij},
\end{equation*}
where $\delta_{ij}$ is the Kronecker symbol.
\end{proof}

By Theorem \ref{a} and Corollary \ref{C-3.1}, we obtain the following corollary.
\begin{corollary}\label{cu}
Suppose that  $(M,G)$ is a real Finsler manifold with a pole $p$ such that its radial flag curvature is bounded from below by a negative constant $-K^2$. Suppose that $\gamma:[0,r] \rightarrow M$ is a normal geodesic with $\gamma(0)=p$ such that $\gamma(r)=x(\neq p)$. Then with respect to the normal coordinates at the point $x$,
\begin{equation*}
\frac{\partial ^2 \rho^2}{\partial x^i \partial x^j}u^iu^j \leq 2(2+\rho K),
\end{equation*}
where $u=u^i\frac{\partial}{\partial x^i}\in T_xM$ is a unit vector.
\end{corollary}
\begin{proof}
Note that $\hat{\nabla} \rho^2=2\rho \hat{\nabla}\rho$ and
$$\Gamma_{i;j}^k(x;\hat{\nabla} \rho^2)=\Gamma_{i;j}^k(x;2\rho\hat{\nabla} \rho)=\Gamma_{i;j}^k(x;\hat{\nabla} \rho).$$ Thus we have
\begin{equation}\label{y}
H(\rho^2)(u,u)= 2(d\rho(u))^2+2\rho H(\rho)(u,u).
\end{equation}

Let $E_1(t),\cdots, E_{n-1}(t), T=\dot{\gamma}$ be the orthogonal vector fields, i.e.,
$$
\langle E_i^H(t),E_j^H(t)\rangle_{\dot{\gamma}}=\delta_{ij},\quad 1\leq i,j\leq n
$$
along $\gamma$.
 If we write $$u=\sum_{i=1}^n(u')^iE_i=\sum_{i=1}^{n-1}(u')^iE_i+(u')^nE_n,$$ we get
\begin{equation}\label{y-1}
d\rho(u)\cdot d\rho(u)=((u')^n)^2 \leq G(u)=1.
\end{equation}
For any given point $(x_0,(\hat{\nabla} \rho)(x_0)) \in \tilde{M}$, there exists a local coordinate system $(x^1,\cdots,x^n,u^1,\cdots,u^n)$ in a neighborhood of $(x_0,(\hat{\nabla} \rho)(x_0))$ such that $G_{ij}(x_0, (\hat{\nabla}\rho)(x_0))=\delta_{ij}$ and $\Gamma_{i;j}^k(x_0;(\hat{\nabla} \rho)(x_0))=0$. By \eqref{y}, \eqref{y-1} and Corollary \ref{C-3.1}, at the point $x_0$, we have
\begin{equation*}\begin{split}
H(\rho^2)(u,u)(x_0)=2+ 2\rho \frac{\partial ^2 \rho}{\partial x^i \partial x^j}u^iu^j\leq 2(2+\rho K).
\end{split}\end{equation*}
\end{proof}

In Riemannian geometry, the classical Gauss lemma is of particular importance. In real  Finsler geometry, there is a similar result (i.e. a Finsler version of the classical Gauss lemma). Now, we recall this theorem. By Theorem 1.6.2 in \cite{abate}, we know that $\exp_p$ is a local $E$-diffeomorphism at the origin. We denote the injectivity radius of $M$ at $p$ by $\mbox{ir}(p)$. Setting
$$B_p(r)=\{x \in M |d(p,x) < r\}, S_p(r)=\{x \in M |d(p,x) = r\}.$$
\begin{theorem}(\cite{abate})\label{g}
Let $G:TM\rightarrow [0,+\infty)$ be a real Finsler metric, fix $p \in M$ and $x \in S_p(r)$. Then $u \in T_xM$ belongs to $T_x(S_p(r))$  iff
$$\langle u^H| T^H\rangle_{T(x)}=0. $$
\end{theorem}

Let $(M,G)$ be a real Finsler manifold with a pole $p$. Fix $p \in M$ and $u\in T_xM$, we denote the distance function from $p$ to $x$ by $\rho(x)$. By Proposition 6.4.2 in \cite{bao}, we know that $\rho^2(x)$ is only a $C^1$ function on $M$. By the classical Hopf-Rinow theorem for a real Finsler metric, there exists a minimizing geodesic $\sigma_u$ connecting $p$ to $x$ such that
$$\rho(x)=d(p,x)=L(\sigma_u).$$
Let $t\rightarrow T(\sigma_u(t))$ be the unit tangent vector to the geodesic $\sigma_u$. Then we know that the gradient of $\rho^2(x)$ is equal to $2\rho \hat{\nabla} \rho$. By Theorem \ref{g} and the fact that $T(x)=\hat{\nabla} \rho$, it follows that $\hat{\nabla} \rho \perp T_x(S_p(\rho))$. So that we have
$$\langle (\hat{\nabla} \rho^2)^H |T^H \rangle_{T(x)}=\langle 2\rho T^H|T^H \rangle_{T(x)}=2\rho(x),$$
where the last step we used the fact that $\langle T^H|T^H \rangle_{T(x)}=1$.

We have proved the following theorem.
\begin{theorem}\label{T-3.3}
Suppose that $(M,G)$ is a real Finsler manifold with a pole $p$. Let $\gamma: [0,r] \rightarrow M$ be a normal geodesic. Then
$$\langle (\hat{\nabla} \rho^2)^H |T^H \rangle_{T(x)}=2\rho(x),$$
where $T=\dot{\gamma}$.
\end{theorem}
\section{The estimation of the distance on complex Finsler manifolds}
\noindent

\par In this section, we obtain the estimations of the distance function on a complex Finsler manifold. We first give a lemma which establishes a relationship between a strongly convex complex Finsler metric $G$ on a complex manifold $M$ of complex dimension $n$  and its associated  real Finsler metric $G^\circ$ on $M$ (considered as a smooth manifold of dimension $2n$).
\begin{lemma}(\cite{abate})\label{b}
Let $G:T^{1,0}M \rightarrow [0, +\infty)$ be a strongly convex complex Finsler metric on a complex manifold $M$. Then
\begin{equation*}\begin{split}
G_{ab}^\circ U^a_1U^b_2=2\mbox{Re}[G_{\alpha\bar{\beta}}V_1^{\alpha}\overline{V_2^{\beta}}+G_{\alpha \beta}V_1^{\alpha}V_2^{\beta}], V_j\in \mathcal{V}_v^{1,0}, U_j= V_j^\circ\in\mathcal{V}_u, \quad \text{for} \quad j=1,2.
\end{split}\end{equation*}
That is
\begin{equation*}
 \forall V,W \in \mathcal{V}_v^{1,0}, \quad \langle V^{\circ}|W^{\circ}\rangle= \mbox{Re}[\langle V,W \rangle + \ll V,W\gg],
\end{equation*}
where $\ll H,K\gg_v=G_{\alpha\beta}(v)H^\alpha K^\beta, \forall H, K \in \mathcal{H}_v^{1,0}$.
\end{lemma}

It follows immediately from Lemma \ref{b} that for a strongly convex complex Finsler metric $G$ on a complex manifold $M$, we have
$$
\frac{1}{2}G_{ab}^\circ u^au^b=G^\circ(u)=G(v)=G_{\alpha\overline{\beta}}v^\alpha \overline{v^\beta}
$$
for any $u=u^i\frac{\partial}{\partial x^i}$ and $v=v^\alpha\frac{\partial}{\partial z^\alpha}=u_\circ=\frac{1}{2}(u-iJu)\in T_z^{1,0}M$ since $G_{\alpha\beta}v^\alpha v^\beta=0$.
\begin{lemma}\label{v}(\cite{li})
Let $f$ be a smooth real-valued function on a strongly convex weakly K\"ahler-Finsler manifold $(M,G)$ and let $X_{\circ}=\frac{1}{2}(X-\sqrt{-1}JX)$ be a vector of type $(1,0)$. Then for every $p \in M$ and for every $X \in T_pM$, we have
\begin{equation*}
Lf(X_\circ,\overline{X_{\circ}})=D^2f(X,X)+D^2f(JX,JX),
\end{equation*}
where $D^2f(X,Y)=X(Yf)-(D_XY)f$ for all real vector fields $X,Y$ on $(M,G)$,  $D$ is the Cartan connection of $G$, and $J$ is the canonical complex structure on $M$.
\begin{remark}
Note that in \cite{cheny}, Chen and Yan proved the above lemma under the assumption that $(M,G)$ is a strongly convex K\"ahler-Berwald manifold.
\end{remark}
\end{lemma}

\begin{theorem}\label{T-4.1}
Suppose that $(M,G)$ is a strongly convex weakly K\"ahler-Finsler manifold with a pole $p$ such that its radial flag curvature is bounded from below by a negative constant $-K^2$. Suppose that $\gamma:[0,r] \rightarrow M$ is a  geodesic with $G(\dot{\gamma})\equiv 1$  such that $\gamma(0)=p$, $\gamma(r)=z \neq p$. Denote the distance function from $p$ to $z$ by $\rho(z)$. Then
\begin{equation*}
\frac{\partial^2 \rho^2}{\partial z^\alpha\partial \bar{z}^\beta}v^\alpha\bar{v}^\beta \leq (2+\rho K),
\end{equation*}
where $v=v^\alpha\frac{\partial}{\partial z^\alpha}=\frac{1}{2}(u-\sqrt{-1}Ju)\in T_z^{1,0}M$ is a unit vector.
\end{theorem}
\begin{proof}
By Lemma \ref{v}, we have
\begin{equation*}\begin{split}
4 \frac{\partial^2 \rho}{\partial z^{\alpha} \partial \bar{z}^{\beta}}v^{\alpha}\bar{v}^{\beta}&=L\rho^2(v,\overline{v})\\
&=D^2\rho^2(u,u)+D^2\rho^2(Ju,Ju).
\end{split}\end{equation*}
Using Corollary \ref{cu} and Lemma \ref{b}, we find that
\begin{equation*}\begin{split}
\frac{\partial^2 \rho^2}{\partial z^{\alpha} \partial \bar{z}^{\beta}}v^{\alpha}\bar{v}^{\beta}
&= \frac{1}{4}\{D^2\rho^2(u,u)+D^2\rho^2(Ju,Ju)\}\\
&\leq \frac{1}{4}\{2(2+\rho K)G^\circ(u)+2(2+\rho K)G^\circ(Ju)\}\\
&=(2+\rho K),
\end{split}\end{equation*}
where the last step we used the fact $G^\circ(u)=G^\circ(Ju)=G(v)=1$.
\end{proof}
\begin{remark}\label{R-4.2}
If $(M,G)$ is a K\"ahler manifold, the condition that $\gamma(r)=z \neq p$ is not necessary. And its radial flag curvature reduces to the radial Riemannian sectional curvature on $(M,G)$. For more details, we refer to \cite{nie}.
\end{remark}
\begin{remark}\label{R-4.3}
If $\triangle$ is the unit disk in $\mathbb{C}$ endowed with the Poincar$\acute{\mbox{e}}$ metric $P$ whose Gaussian curvature is $-4$, and we denote $\varrho(\zeta)$ the distance function from $0$ to $\zeta\in\triangle$. Then by Theorem \ref{T-4.1} or Lemma 5.3 in \cite{nie}, we have
\begin{equation}
\frac{\partial^2\varrho^2(\zeta)}{\partial \zeta\partial\overline{\zeta}}\leq 2[1+2\varrho(\zeta)],\quad\forall \zeta\in\triangle.
\end{equation}
\end{remark}
\begin{theorem} \label{T-4.2}
Suppose that $(M,G)$ is a strongly convex complex Finsler manifold with a pole $p$. Suppose that $\gamma:[0,r] \rightarrow M$ is a geodesic with $G(\dot{\gamma})=1$ such that $\gamma(0)=p$ and $\gamma(r)=z$. Then
$$\rho(z)=\frac{1}{2}\langle ((\hat{\nabla} \rho^2(z))_\circ)^H, (T(z))^H\rangle_{T(z)},$$
where $T(z)=\dot{\gamma}$ and $G((T(z)))=1$.
\end{theorem}
\begin{proof}
By Proposition 6.4.2 in \cite{bao} and the fact that $p$ is a pole, we know that $\rho(z)$ is equal to the length of minimizing geodesic connecting $p$ and $z$. Furthermore, $\rho^2(z)$ is a $C^1$ function on $M$. By Theorem \ref{T-3.3} and the fact $G^\circ((T(z))^\circ)=G(T(z))=1$ , we have
$$\rho(z)=\frac{1}{2}\langle ((\hat{\nabla} \rho^2(z))^\circ)^H|((T(z))^{\circ})^H \rangle_{(T(z))^\circ}.$$
Therefore, we know
$$\hat{\nabla} \rho^2(z)=2 \rho \hat{\nabla} \rho=2\rho(z) (T(z)+\overline{T(z)}). $$
By Lemma \ref{b}, we have
\begin{equation*}\begin{split}
\rho(z)&=\frac{1}{2}\langle ((\hat{\nabla} \rho^2(z))^\circ)^H|((T(z))^{\circ})^H \rangle_{(T(z))^\circ}\\
&=\frac{1}{2}\mbox{Re}[\langle ((\hat{\nabla} \rho^2(z))_\circ)^H, (T(z))^H\rangle_{T(z)}]\\
&=\frac{1}{2}\langle ((\hat{\nabla} \rho^2(z))_\circ)^H, (T(z))^H\rangle_{T(z)},
\end{split}\end{equation*}
where the last step we used the fact that $\langle ((\hat{\nabla} \rho^2(z))_\circ)^H, (T(z))^H\rangle_{T(z)}$ is a real number.
\end{proof}
\begin{remark}\label{R-4.4}
If $\triangle$ is the unit disk in $\mathbb{C}$ endowed with the Poincar$\acute{\mbox{e}}$ metric $P$ whose Gaussian curvature is $-4$, and $\varrho(\zeta)$ is the distance function from $0$ to $\zeta\in\triangle$. Then by Theorem \ref{T-4.2} or Lemma 5.3 in \cite{nie},
\begin{equation}
2\varrho(\zeta)=\langle (\hat{\nabla} \varrho^2(\zeta))_\circ,T(\zeta)\rangle.
\end{equation}
\end{remark}
\section{Some lemmas}
\setcounter{equation}{0}
\noindent
\par
 Note that the square of distance function on a strongly convex weakly K\"ahler-Finsler manifold with a pole $p$ is not necessary smooth. Thus in order to prove the Theorem \ref{T-1.3}, we need to overcome this difficulty. We denote by $\rho(z)$ the distance function from $p$ to $z$ in a strongly convex weakly K\"ahler-Finsler manifold with a pole. By Proposition 6.4.2 in \cite{bao}, we know that $\rho^2(z)$ is not smooth  at $p$. If $p$ is a pole of $M$, we know that $\rho^2(z)$ is only smooth outside $p$. To overcome this inconvenience, we need the following lemma.

\begin{lemma}\label{L-5.2}
Suppose that $(M,G)$ is a strongly convex weakly K\"ahler-Finsler manifold with a pole $p$ such that its radial flag curvature is bounded from below by a negative constant $-K^2$. Then for any sufficiently small $\varepsilon>0$, there exists a smooth function $\phi$ on $M$ such that $\phi \equiv \rho^2$ on $\{z \in M| f(z) \geq \varepsilon\},$ and its derivatives of any order vanish at $p$, and on $\{z \in M| f(z) < \varepsilon\},$
$$\frac{\partial^2 \phi}{\partial z^{\alpha}\partial \bar{z}^{\beta}} v^{\alpha}\bar{v}^{\beta} \leq C_1 \quad \text{and} \quad \frac{1}{2}\langle ((\hat{\nabla} \phi(z))_\circ)^H, (T(z))^H\rangle_{T(z)}=\Big|\frac{\partial \phi}{\partial z^{\alpha}}v^{\alpha}\Big|^2 \leq C_2,$$
where $C_1, C_2$ are constants and independent on $\rho$.

On the other hand, on $\{z \in M| f(z) \geq \varepsilon\},$
\begin{equation*}\begin{split}
\frac{\partial^2 \phi}{\partial z^{\alpha} \partial \bar{z}^{\beta}}v^{\alpha}\bar{v}^{\beta} \leq (2+\rho K) \quad \text{and} \quad \rho(z)=\frac{1}{2}\langle ((\hat{\nabla} \phi(z))_\circ)^H, (T(z))^H\rangle_{T(z)},
\end{split}\end{equation*}
where $v= v^{\alpha}\frac{\partial}{\partial x^\alpha}=\frac{1}{2}(u-\sqrt{-1}Ju)\in T_z^{1,0}M$ is a unit vector and $T=\dot{\gamma}$.
\end{lemma}
\begin{proof}
We use the ideas in \cite{cheny} to prove the above lemma.
For any small $\varepsilon$ such that $0<\varepsilon<1$, one can define a smooth function $g_1(t)$ on $\mathbb{R}^1$ as follows:
\begin{equation*}
g_{1}(t)=\left\{\begin{array}{ll}
e^{\frac{1}{(t-\varepsilon)(t+3 \varepsilon)}} & \text { if }\quad t \in(-3 \varepsilon, \varepsilon) \\
0 & \text { otherwise }
\end{array}\right.
\end{equation*}
Let $H_{1}(t)=\int_{-\infty}^{t} g_{1}(s) d s / \int_{-\infty}^{+\infty} g_{1}(s) d s$. Then $H_{1}(t)$ is still a smooth function on $\mathbb{R}^{1}$ and equals 0 if $t \leq-3 \varepsilon$ and equals 1 if $t \geq \varepsilon .$ By Proposition 6.4.2 in \cite{bao} and the fact that $p$ is a pole of $M$, $f:=\rho^2(z)$ is a continuous function on $M$, and it reaches its minimum value $0$ at $p$. Define $f_{1}(z)=H_{1}(f(z))(f(z)-\varepsilon)+\varepsilon .$ Then $f_{1}$ satisfies the following properties:

$i)$ $f_{1}$ is continuous on $M$ and smooth outside $p$, and $f_{1}=f$ if $f \geq \varepsilon$.

$ii)$ For any $0<f(z)<\varepsilon,(f_{1})_{z}$ and $(f_{1})_{\bar{z}}$ have the same sign with $f_{z}$ and $f_{\bar{z}}$, respectively, which means $f_{1}$ also reaches its minimum value at $z=0$. In fact,
$$(f_{1})_{z}=[(H_{1})_{f}(f-\varepsilon)+H_{1}] f_{z},$$
and if $t \in(0, \varepsilon)$, we have  $H_{1}>\frac{3}{4}$ and $(H_{1})_{f}(\varepsilon-t)<\frac{1}{2}$ . It clear holds $t \geq \varepsilon$ since $f_{1}=f$ in this case.

$iii)$ For any $0<f(z)<\varepsilon,$ we have
\begin{equation*}\begin{split}
 \frac{1}{2}\langle ((\hat{\nabla} f_1(z))_\circ)^H, (T(z))^H\rangle_{T(z)}&=\Big|\frac{\partial f_1}{\partial z^{\alpha}}v^{\alpha}\Big|\\
 &=\Big|[(H_{1})_{f}(f-\varepsilon)+H_{1}]\frac{\partial f}{\partial z^{\alpha}}v^{\alpha}\Big|\\
 &\leq M_0\Big|\frac{\partial f}{\partial z^{\alpha}}v^{\alpha}\Big|\\
 & \leq M_0=C_1,
 \end{split}\end{equation*}
 where $M_0$ is the maximum value of the continuous function $|[(H_{1})_{f}(f-\varepsilon)+H_{1}]|$ on $\{f: 0 \leq f\leq 1\}\supset\overline{\{f: 0 <f <\varepsilon\}}$ and $M_0$ is independent on $\rho$ and the last step we used the Theorem \ref{T-4.2}.

 $iv)$ For any $0<f(z)<\varepsilon,$ we have
 \begin{equation*}\begin{split}
 \frac{\partial^2 f_1}{\partial z^{\alpha} \partial \bar{z}^{\beta}}v^{\alpha}\bar{v}^{\beta}&=(H_1)_{ff}(f-\varepsilon)\Big|\frac{\partial^2 f}{\partial z^{\alpha}}v^{\alpha}\Big|^2+
 (H_1)_f(f-\varepsilon)\frac{\partial^2f}{\partial z^{\alpha}\partial \bar{z}^{\beta}}v^\alpha\bar{v}^\beta\\
 &+2(H_1)_f\Big|\frac{\partial f}{\partial z^{\alpha}}v^{\alpha}\Big|^2 +H_1\frac{\partial^2 f}{\partial z^{\alpha} \partial \bar{z}^\beta}v^\alpha \bar{v}^{\beta}\\
 &\leq (M_1+M_3)\Big|\frac{\partial f}{\partial z^{\alpha}}v^{\alpha}\Big|^2+(M_2+M_4)\frac{\partial^2 f}{\partial z^{\alpha} \partial \bar{z}^\beta}v^\alpha \bar{v}^{\beta}\\
 & \leq (M_1+M_3)+(M_2+M_4)(2+K)=C_2,
 \end{split}\end{equation*}
 where $M_1, M_2, M_3, M_4$ are the maximum value of the continuous functions $(H_1)_{ff}(f-\varepsilon), (H_1)_f(f-\varepsilon), 2(H_1)_f, H_1$ on $\{f: 0 \leq f\leq 1\}\supset\overline{\{f: 0 <f <\varepsilon\}}$, respectively and $M_1, M_2, M_3, M_4$ are independent on $\rho$ and the last step we used Theorem \ref{T-4.1} and \ref{T-4.2}.

 Now we define $f_{2}$ from $f_{1}$. Let $\varepsilon_{1}=f_{1}\left(z_{1}\right)$ for any $z_{1}$ satisfying $f\left(z_{1}\right)=\frac{\varepsilon}{2}$. Let
$$
g_{2}(t)=\left\{\begin{array}{ll}
e^{\frac{1}{\left(t-\varepsilon_{1}\right)\left(t+3 \varepsilon_{1}\right)}} & \text { if }\quad t \in\left(-3 \varepsilon_{1}, \varepsilon_{1}\right) \\
0 & \text { otherwise }
\end{array}\right.
$$

 Define $H_{2}(t)=\int_{-\infty}^{t} g_{2}(s) d s / \int_{-\infty}^{+\infty} g_{2}(s) d s$. Then $H_{2}(t)$ is still a smooth function on $\mathbb{R}^{1}$ and equals 0 if $t \leq-3 \varepsilon_{1}$ and equals 1 if $t \geq \varepsilon_{1}$. Define $f_{2}(z)=H_{2}\left(f_{1}(z)\right)\left(f_{1}(z)-\varepsilon_{1}\right)+\varepsilon_{1}$.
Then $f_2$ have similar properties i)-iv). By induction, we can define $f_{n+1}$ from $f_{n}$. Let $\varepsilon_{n}=f_{n}\left(z_{n}\right)$ for any $z_{n}$ satisfying $f\left(z_{n}\right)=\frac{\varepsilon}{2^{n}}$. Let
$$
g_{n+1}(t)=\left\{\begin{array}{ll}
e^{\frac{1}{\left(t-\varepsilon_{n}\right)\left(t+3 \varepsilon_{n}\right)}} & \text { if }\quad t \in\left(-3 \varepsilon_{n}, \varepsilon_{n}\right) \\
0 & \text { otherwise }
\end{array}\right.
$$
Define $H_{n+1}(t)=\int_{-\infty}^{t} g_{n+1}(s) d s / \int_{-\infty}^{+\infty} g_{n+1}(s) d s$. Then $H_{n+1}(t)$ is still a smooth function on $R^{1}$. It equals to $0$ if $t \leq-3 \varepsilon_{n}$ and equals to $1$ if $t \geq \varepsilon_{n}.$ Define $f_{n+1}(z)=$ $H_{n+1}\left(f_{n}(z)\right)\left(f_{n}(z)-\varepsilon_{n}\right)+\varepsilon_{n}.$ Then $f_{n+1}$ have similar properties i)-iv).

Now define $\phi(z) \doteq \lim _{n \rightarrow \infty} f_{n}(z).$ It is easy to check that $\phi$ is well-defined and it is the desired function we want.

By Theorem \ref{T-4.1} and \ref{T-4.2}, on $\{z \in M| f(z) \geq \varepsilon\},$ we have,
\begin{equation*}\begin{split}
\frac{\partial^2 \phi}{\partial z^{\alpha} \partial \bar{z}^{\beta}}v^{\alpha}\bar{v}^{\beta} \leq (2+\rho K) \quad \text{and} \quad \rho(z)=\frac{1}{2}\langle ((\hat{\nabla} \phi(z))_\circ)^H, (T(z))^H\rangle_{T(z)}.
\end{split}\end{equation*}
\end{proof}
For the convenience of proving Theorem \ref{T-1.3}, we rewrite the lemma \ref{L-5.2} into the following lemma.
\begin{lemma}\label{L-5.3}
Suppose that $(M,G)$ is a strongly convex weakly K\"ahler-Finsler manifold with a pole $p$ such that its radial flag curvature is bounded from below by a negative constant $-K^2$. Then for any sufficiently small $\varepsilon>0$, there exists a smooth function $\phi$ on $M$ such that $\phi \equiv \rho^2$ on $\{z \in M| f(z) \geq \varepsilon\},$ and on $M$,
$$\frac{\partial^2 \phi}{\partial z^{\alpha}\partial \bar{z}^{\beta}} v^{\alpha}\bar{v}^{\beta} \leq \max \{C_1, 2+\rho K\} \quad \text{and} \quad \frac{1}{2}\langle ((\hat{\nabla} \phi(z))_\circ)^H, (T(z))^H\rangle_{T(z)}\leq \max \{C_2,\rho\},$$
where $C_1, C_2$ are constants and independent on $\rho$ and $v= v^{\alpha}\frac{\partial}{\partial x^\alpha}=\frac{1}{2}(u-\sqrt{-1}Ju)\in T_z^{1,0}M$ is a unit vector and $T=\dot{\gamma}$.
\end{lemma}
\section{The proof of Theorem \ref{T-1.3}}
\setcounter{equation}{0}
\noindent
\par In this section, we prove the following main theorem (i.e., Theorem \ref{T-1.3}) of this paper. Firstly, we recall Theorem \ref{T-1.3} here for convenience.
\begin{theorem}\label{T- 6.1a}
Suppose that $(M,G)$ is a strongly convex weakly K$\ddot{a}$hler-Finsler manifold with a pole $p$ such that its radial flag curvature is bounded from below and its holomorphic sectional curvature $K_G$ is bounded from below by a constant $K_1 \leq 0$. Suppose that $(N,H)$ is a strongly pseudoconvex complex Finsler manifold with holomorphic sectional curvature $K_H$ bounded from above by a constant $K_2<0$. Let $f:M \rightarrow N$ be a holomorphic mapping. Then
\begin{equation}\label{EQ-S}
(f^*H)(z;dz) \leq \frac{K_1}{K_2}G(z;dz).
\end{equation}
\end{theorem}
\begin{proof}
The key point of the proof is the construction of the auxiliary function \eqref{aaf} and then using maximum principle. This is essentially different from the case when $G$ is a Hermitian quadratic metric (e.g. K\"ahler metric or Hermitian metric) on $M$ since the square of distance function $\rho^2(z)$ associated to $G$ is smooth over the whole $M$ if $G$ is a Hermitian metric while $\rho^2(z)$ is only smooth over $M-\{p\}$ if $G$ is a non-Hermitian quadratic metric and $p$ is a pole of $M$ (see Proposition 6.4.2 in \cite{bao}). In the non-Hermitian quadratic case we need the function $\phi$ in Lemma \ref{L-5.3} to overcome this inconvenience.

Let $B_a(p)$ be a closed geodesic ball in $(M,G)$ with its center at $p$ and radius of $a\in(0,+\infty)$. Let $\Delta$ be a unit disk with Poincar\'e metric $P$.
And let $B_b$ be a closed geodesic ball in $(\Delta, P)$ with its center at $0$ and radius of $b\in(0,+\infty)$. We denote the distance function from $0$ to $\zeta$ on $\Delta$ by $\varrho(\zeta)$, and the distance function from $p$ to $z$ on $M$ by $\rho(z)$. Suppose that $\varphi$ is any holomorphic mapping from $\Delta$ into $M$ such that $\varphi(B_b) \subset B_a(p)$ and $\varphi(0)=p$.

The pull-back metric on $\Delta$ of $G$ on $M$ by the holomorphic mapping $\varphi:\Delta\rightarrow M$ is given by
 \begin{equation*}
 (\varphi^*G)(\zeta)= \lambda^2(\zeta)d\zeta d\bar{\zeta}.
 \end{equation*}
 Here we have denoted
 \begin{equation*}
 \lambda^2(\zeta):=G(\varphi(\zeta);\varphi'(\zeta)).
 \end{equation*}

 Note that since $G(\varphi'(\zeta))\equiv1$, we have $\varphi'(\zeta)\neq 0$. Thus $\lambda(\zeta) >0$.
 Now let
 \begin{equation*}
 (f\circ \varphi)^*H(\zeta) =\sigma^2(\zeta) d\zeta d\bar{\zeta}
 \end{equation*}
 be the pull-back metric on $\Delta$ of $H$ on $N$ by the holomorphic mapping $f\circ \varphi:\Delta\rightarrow N$. Here, we have denoted
 \begin{equation*}
 \sigma^2(\zeta):= H((f\circ\varphi)(\zeta);(f\circ\varphi)'(\zeta)).
 \end{equation*}

 i) If $(f\circ\varphi)'(\zeta)= 0$, then \eqref{EQ-S} holds obviously with $\sigma(0)=0$.

 ii)  Suppose that $(f\circ\varphi)'(\zeta)\neq 0$.

 We define the following auxiliary function:
\begin{equation}
\Phi(\zeta):=[a^2-\phi(\varphi(\zeta))]^2[b^2-\varrho^2(\zeta)]^2\frac{\sigma^2(\zeta)}{\lambda^2(\zeta)},\label{aaf}
\end{equation}
where $\phi$ is defined in Lemma \ref{L-5.3} for a function $\rho^2$.
It is clear that $\Phi(\zeta)\geq 0$ for any $\zeta\in B_b$.

 From Lemma \ref{L-5.3}, it follows that  $\phi(z)$ is a smooth function on $M$. Thus the function $\Phi(\zeta)$ defined by \eqref{aaf} is  smooth for $\zeta\in B_b$. Moreover, $\Phi(\zeta)$ attains its maximum at some interior point $\zeta=\zeta_0\in B_b$ since $\Phi(\zeta) \rightarrow 0$ as $\varrho(\zeta)\rightarrow b$, or equivalently $\zeta$ tends to the boundary  $\partial B_b$ of $B_b$. Thus it suffice for us to seek an upper bound of $\Phi(\zeta_0)$ for an arbitrary $\varphi(\zeta)$ satisfying $(f\circ\varphi)'(\zeta)\neq 0$. We want to use the maximum principle.

 In the following, in order to abbreviate expression of formulas. We denote $\varphi(\zeta)=z=(z^1,\cdots,z^n),\varphi(\zeta_0)=z_0=(z_0^1,\cdots,z_0^n)$, $\varphi'(\zeta)=v=(v^1,\cdots,v^n)$ and $\varphi'(\zeta_0)=v_0=(v_0^1,\cdots,v_0^n)$.

Since $\Phi(\zeta)$ is smooth for $\zeta\in B_b$ and attains its maximum at the interior point $\zeta_0\in B_b$, it necessary that at the point $\zeta=\zeta_0$:
\begin{equation}
0=\frac{\partial}{\partial \zeta} \log \Phi(\zeta)\quad\mbox{and}\quad
0 \geq \frac{\partial^2}{\partial \zeta \partial \bar{\zeta}} \log\Phi(\zeta).\label{sod}
\end{equation}

Substituting \eqref{aaf} in the second inequality in \eqref{sod}, we have
\begin{equation}\begin{split}\label{EQ-6.2}
0 \geq &2\frac{\partial^2}{\partial \zeta \partial \bar{\zeta}}\log[a^2-\phi(z)]
+\frac{\partial^2 }{\partial \zeta \partial \bar{\zeta}}\log \sigma^2(\zeta)-\frac{\partial^2 }{\partial \zeta \partial \bar{\zeta}}\log\lambda^2(\zeta)+2\frac{\partial^2 }{\partial \zeta \partial \bar{\zeta}}\log [b^2-\varrho^2(\zeta)]\\
=&-2[a^2-\phi(z)]^{-1}\frac{\partial^2\phi(z)}{\partial z^\alpha\partial\bar{z}^\beta}v^{\alpha}\bar{v}^{\beta}-2[a^2-\phi(z)]^{-2}\Big|\frac{\partial\phi(z)}{\partial z^\alpha}v^{\alpha}\Big|^2\\
&+\frac{\partial^2 }{\partial \zeta \partial \bar{\zeta}}\log \sigma^2(\zeta)-\frac{\partial^2 }{\partial \zeta \partial \bar{\zeta}}\log \lambda^2(\zeta)\\
&-2[b^2-\varrho^2(\zeta)]^{-1}\frac{\partial^2 \varrho^2(\zeta)}{\partial \zeta \partial \bar{\zeta}}-2[b^2-\varrho^2(\zeta)]^{-2}\Big|\frac{\partial \varrho^2(\zeta)}{\partial \zeta}\Big|^2.
\end{split}\end{equation}
By Lemma \ref{L-2.1} and the curvature assumptions of $G$ and $H$ in Theorem \ref{T- 6.1a}, we have
\begin{equation}
\frac{\partial^2 }{\partial \zeta \partial \bar{\zeta}}\log \sigma^2(\zeta)\geq -2K_2\sigma^2(\zeta),\quad \frac{\partial^2 }{\partial \zeta \partial \bar{\zeta}}\log \lambda^2(\zeta)\leq -2K_1\lambda^2(\zeta).\label{cc1}
\end{equation}
By Lemma \ref{L-5.3}, we have
\begin{equation}\begin{split}\label{EQ-6.3}
\frac{\partial^2\phi(z)}{\partial z^\alpha\partial \bar{z}^\beta}v^{\alpha}\bar{v}^{\beta}
 \leq \max\{(2+\rho K),C_1\}
\leq \max\{(2+aK),C_1\}:=A,
\end{split}\end{equation}
where in the last step we used the inequality $\rho(z) \leq a$.

By the Remark \ref{R-4.3}, at $\zeta=\zeta_0$, we have
\begin{equation}\label{A-7}
\frac{\partial^2 \varrho^2(\zeta)}{\partial \zeta \partial \bar{\zeta}} \leq 2[1+2\varrho(\zeta_0)]\leq 2(1+2b),
\end{equation}
since $\varrho(\zeta_0)\leq b$.

In order to estimate the first order term of $\phi(z)$ and $\varrho^2(\zeta)$ in \eqref{EQ-6.2}, we use normal coordinates. Since $M$ is a strongly convex weakly K\"ahler-Finsler manifold, we can choose coordinates around $(z_0,v_0)$ such that at the point $(z_0,v_0)$, we have
$$G_{\alpha \bar{\beta}}(z_0,v_0)=\delta_{\alpha \beta}, \quad 1 \leq \alpha, \beta \leq n.$$

Thus by Lemma \ref{L-5.3},  at $\zeta=\zeta_0$, we have
\begin{equation}\begin{split}\label{EQ-6.4}
\Big|\frac{\partial\phi(z_0)}{\partial z^\alpha}v_0^{\alpha}\Big|
&\leq |\langle (\hat{\nabla}(a^2-\phi(z_0)))_\circ,T(z_0)\rangle|\\
&=|\langle (\hat{\nabla}(\phi(z_0)))_\circ,T(z_0)\rangle|=\max\{2\rho(z_0),C_2\} \\
&\leq \max\{2a,C_2\}:=B,
\end{split}\end{equation}
where the last step we used the fact that $\rho(z_0) \leq a$.

For the same reasons as \eqref{EQ-6.4}, at $\zeta=\zeta_0$, that is Remark \ref{R-4.4}, we have
\begin{equation}\label{A-8}
\Big|\frac{\partial \varrho^2(\zeta)}{\partial \zeta}\Big| \leq 2b.
\end{equation}

Substituting \eqref{cc1}, \eqref{EQ-6.3}, \eqref{A-7}, \eqref{EQ-6.4},  \eqref{A-8} into \eqref{EQ-6.2}, we have (at $\zeta=\zeta_0$),
\begin{equation*}
0 \geq -K_2\sigma^2(\zeta_0)+K_1\lambda^2(\zeta_0)-\frac{A}{2(a^2-\phi(z_0))}-\frac{B^2}{[a^2-\phi(z_0)]^2}-
\frac{2(1+2b)}{b^2-\varrho^2(\zeta_0)}-\frac{4b^2}{[b^2-\varrho^2(\zeta_0)]^2}.
\end{equation*}
Rearranging terms, we get
\begin{equation*}\begin{split}
&\frac{1}{\lambda^2(\zeta_0)}\Big\{A[a^2-\phi(z_0)][b^2-\varrho^2(\zeta_0)]^2+B^2[b^2-\varrho^2(\zeta_0)]^2\\
&+2(1+2b)[b^2-\varrho^2(\zeta_0)][a^2-\phi(z_0)]^2+4b^2[a^2-\phi(z_0)]^2\Big\}-K_1[a^2-\phi(z_0)]^2[b^2-\varrho^2(\zeta_0)]^2 \\
&\geq -K_2\frac{\sigma^2(\zeta_0)}{\lambda^2(\zeta_0)}[a^2-\phi(z_0)]^2[b^2-\varrho^2(\zeta_0)]^2=-K_2\Phi(\zeta_0)\\
&\geq-K_2 \Phi(\zeta)=-K_2[a^2-\phi(\varphi(\zeta))]^2[b^2-\varrho^2(\zeta)]^2\frac{\sigma^2(\zeta)}{\lambda^2(\zeta)}
\end{split}\end{equation*}
for any $\zeta\in B_b$. Now divided by $a^4b^4$ on both side of the above inequality and then letting $a \rightarrow +\infty$ and $b \rightarrow +\infty$, respectively, we obtain
$$\frac{\sigma^2(\zeta)}{\lambda^2(\zeta)} \leq \frac{K_1}{K_2}$$
for any holomorphic mapping $\varphi$ from $\Delta$ into $M$ satisfying $(f\circ\varphi)'(\zeta)\neq 0$.\\
This completes the proof.
\end{proof}
\begin{theorem}\label{T-6.1b}
Suppose that $(M,G)$ is a strongly complete convex weakly K$\ddot{a}$hler-Finsler manifold such that its radial flag curvature is bounded from below and its holomorphic sectional curvature $K_G$ is bounded from below by a constant $K_1 \leq 0$. Suppose that $(N,H)$ is a strongly pseudoconvex complex Finsler manifold with holomorphic sectional curvature $K_H$ bounded from above by a constant $K_2<0$. Let $f:M \rightarrow N$ be a holomorphic mapping. Then
\begin{equation}\label{EQ-S}
(f^*H)(z;dz) \leq \frac{K_1}{K_2}G(z;dz).
\end{equation}
\end{theorem}

\begin{proof}
If $(M,G)$ is a strongly complete weakly K\"ahler-Finsler manifold without cut points, then the Lemma \ref{L-5.3} previously used in the proof of Theorem \ref{T-1.3} still hold.

 If $(M,G)$ is a strongly complete weakly K\"ahler-Finsler manifold with cut points. The proof essentially goes the same lines as in Chen-Cheng-Lu \cite{chen}. With the notations of Theorem \ref{T- 6.1a} and its proof. We use the notations in Theorem $6.1$ and its proof. Let $p$ be an arbitrary point which is not a cut point and $z_{0} \in M$ at which $\Phi(\zeta)$ attains its maximum value. That is, $\varphi\left(\zeta_{0}\right)=z_{0}$. Since $\left(M, d s_{M}^{2}\right)$ is complete, thus Hopf-Rinow theorem for the Finsler metric, there exists a minimizing geodesic $\gamma:[0,1] \rightarrow M$ joining $p$ and $z_{0}$ such that $\gamma(0)=p$ and $\gamma(1)=z_{0} .$ If there is a $t_{0} \in(0,1)$ such that $\gamma\left(t_{0}\right)=p_{1}$ is first cut point to the point $z_{0}$ along the inversely directed geodesic $\gamma_{1}=\gamma(1-t)$ for all $t \in[0,1] .$ Let $\varepsilon>0$ be a given sufficiently small number such that $t_0+\varepsilon<1$, then it clear that $z_{0}$ is not a cut point of $\rho\left(t_{0}+\varepsilon\right)$ with respect to the geodesic $\gamma_{1}$. Define $\tilde{\rho}(p, z):=\rho\left(p, \gamma\left(t_{0}+\varepsilon\right)\right)+\rho\left(\gamma\left(t_{0}+\varepsilon\right), z\right)$. Then using
the triangle inequality, we have
$$
\rho(p, z) \leqslant \tilde{\rho}(p, z) \quad \text { and } \quad \rho\left(p, z_{0}\right)=\tilde{\rho}\left(p, z_{0}\right) .
$$
So that
$$
\Phi_1(\zeta)=\left[a^{2}-\phi_1(\varphi(\zeta))\right]^{2}\left[b^{2}-\varrho^{2}(\zeta)\right]^{2} \frac{\sigma^{2}(\zeta)}{\lambda^{2}(\zeta)}
$$
is smooth at the point $\zeta_{0}$ and we have
$$
\Phi_1(\zeta)=\left[a^{2}-\phi_1(\varphi(\zeta))\right]^{2}\left[b^{2}-\varrho^{2}(\zeta)\right]^{2} \frac{\sigma^{2}(\zeta)}{\lambda^{2}(\zeta)} \leqslant \Phi(\zeta)
$$
and
$$
\Phi_1(\zeta_0)=\left[a^{2}-\phi_1\left(\varphi\left(\zeta_{0}\right)\right)\right]^{2}\left[b^{2}-\varrho^{2}\left(\zeta_{0}\right)\right]^{2} \frac{\sigma^{2}\left(\zeta_{0}\right)}{\lambda^{2}\left(\zeta_{0}\right)}=\Phi\left(\zeta_{0}\right),
$$
where $\phi_1$ is defined in Lemma \ref{L-5.2} for a function $\tilde{\rho}^2$  near $z_0$ . Now by passing the discussion of $\tilde{\rho}(p, z)$ to $\rho\left(\gamma\left(t_{0}+\varepsilon\right), z\right)$, the remaining proof goes the same lines as Theorem \ref{T- 6.1a}. This completes the proof.
\end{proof}
\section{Applications of Theorem \ref{T-1.3}}
\noindent

In this section, we give some applications of Theorem \ref{T-1.3}. In \cite{kobayashi1,kobayashi2}, Kobayashi introduced lots of results about holomorphic mappings between complex manifolds, including various types of  Schwarz lemmas. By Theorem \ref{T-1.3}, we obtain the following theorem.
\begin{corollary}\label{C-6.1}
Suppose $(M,G)$ is a strongly convex weakly K\"ahler-Finsler manifold with a pole $p$ such that its holomorphic sectional curvature is non-negative and its radial flag curvature is bounded from below. Suppose that $(N,H)$ is a strongly pseudoconvex complex Finsler manifold such that its holomorphic sectional curvature is bounded from above by a constant $K_2 <0$. Then any holomorphic mapping $f$ from $M$ into $N$ is a constant.
\end{corollary}

\par Now we recall the concept of a complex Minkowski space.
\begin{definition}(\cite{abate})
A complex Minkowski space is $\mathbb{C}^n$ endowed with a complex Finsler metric $G: \mathbb{C}^n \times \mathbb{C}^n\cong T^{1,0}\mathbb{C}^n \rightarrow [0,+ \infty)$ given by
$$ \forall p \in \mathbb{C}^n, \quad \forall v \in T^{1,0}\mathbb{C}^n \cong \mathbb{C}^n, \quad G(p;v) = ||v||^2,$$
where $\|\cdot\|: \mathbb{C}^n \rightarrow [0,+\infty)$ is a complex norm (with strongly convex unit ball) on $\mathbb{C}^n$.
\end{definition}
\begin{remark}
If $\|\cdot\|$ is not the norm associated to a Hermitian inner product, then $F$ does not come from a Hermitian metric.
\end{remark}

Let $(\mathbb{C}^n, G)$ be a complex Minkowski space, then
$$\Gamma_{\beta;\mu}^{\alpha}=0.$$
 Thus a complex Minkowski space is necessary a K\"ahler-Finsler manifold. It is obvious that a complex Minkowski space is also a real Minkowski space so that its horizontal flag curvature vanishes identically. By the definition of holomorphic sectional curvature, it follows that the holomorphic sectional curvature of a complex Minkowski space vanishes identically. In the following, we prove that $\exp_p$ is an $E$-diffeomorphism at the origin. Now we introduce the Cartan-Hadamard theorem in a real Finsler manifold.
\begin{theorem}(\cite{abate})\label{T-6}
Let $(M,G)$ be a complete real Finsler manifold, and fix $p \in M$. Assume that the Morse index form $I_0^r$ is positive definite on $\mathcal{X}_0[0,r]$ for all $r>0$ and along every radial normal geodesic issuing from $p$ (e.g., assume that the horizontal flag curvature is non-positive.) Then  $\exp_p: T_pM \rightarrow M$ is a covering map, smooth outside the origin. In particular, if $(M,G)$ is simply connected then $\exp_p$ is an $E$-diffeomorphism at the origin.
\end{theorem}
\begin{theorem}\label{T-6.1}
Suppose that $(\mathbb{R}^n,G)$ is a real Minkowski space, and fix $p \in M$. Then $\exp_p$ is an $E$-diffeomorphism at the origin.
\end{theorem}
\begin{proof}
By the Hopf-Rinow theorem for a real Finsler metric (see Theorem 1.6.9 in \cite{abate}), $(\mathbb{R}^n,G)$ is compete. Since the horizontal flag curvature of $(\mathbb{R}^n,G)$ vanishes identically. This together with Theorem \ref{T-6} and the fact that $\mathbb{R}^n$ is simply connected yields that $\exp_p$ is an $E$-diffeomorphism at the origin.
\end{proof}
By Corollary \ref{C-6.1} and Theorem \ref{T-6.1}, we have the following corollary.
\begin{corollary}
Suppose that $(\mathbb{C}^n, G)$ is a complex Minkowski space. Suppose that $(N,H)$ is a strongly pseudocnvex complex Finsler manifold such that its holomorphic sectional curvature is bounded from above by a negative constant $K_2$. Then any holomorphic mapping $f$ from $\mathbb{C}^n$ into $N$ is a constant.
\end{corollary}
\begin{remark}
If $H$ comes from a Hermitian metric on $N$, then the above corollary is obviously true.
\end{remark}

 Now we consider that $(N,H)$ is a Hermitian manifold. By Theorem \ref{T-6.1b}, we obtain the following corollary.
\begin{corollary}\label{C-6.2}
Suppose that $(M,G)$ is a complete strongly convex weakly K\"ahler-Finsler manifold such that its radial flag curvature is bounded from below and its holomorphic sectional curvature $K_G$ is bounded from below by a constant $K_1 \leq 0$. Suppose that $(N,ds^2_N)$ is a Hermitian manifold such that its holomorphic sectional curvature $K_H$ is bounded from above by a constant $K_2<0$. Then for any holomorphic mapping $f$ from $M$ into $N$,
$$f^*ds^2_N\leq \frac{K_1}{K_2}G(z;dz).$$
In particular, $K_1\geq 0$, then any holomorphic mapping $f$ from $M$ into $N$ is a constant.
\end{corollary}

\begin{remark}
In \cite{grauert}, we know that every compact Riemann surface $X$ of genus $\geq 2$ carries a Hermitian metric $ds_X^2$ with holomorphic sectional curvature $K \leq -1$. Therefore, there exist many Hermitian manifolds which satisfy the conditions of Corollary \ref{C-6.2}.
\end{remark}
The following corollary is an application of Corollary \ref{C-6.2}.
\begin{corollary}
Suppose that $(\mathbb{C}^n, G)$ is a complex Minkowski space. Suppose that $(N,H)$ is a Hermitian manifold such that its holomorphic sectional curvature is bounded from above by a negative constant $K_2$. Then any holomorphic mapping $f$ from $\mathbb{C}^n$ into $N$ is a constant.
\end{corollary}
\bigskip
{\bf Acknowledgement:}\ {\small This work is supported by the National Natural Science Foundation of China (No. 12071386, No. 11671330, No. 11271304, No. 11971401)}.

\end{document}